\documentclass[12pt]{article}

\usepackage[letterpaper,
            left=1.15in,
            right=1.15in,
            top=1.25in,
            bottom=1.0in]{geometry}

\usepackage[utf8]{inputenc}

\usepackage[caption=false,font=normalsize,labelfont=sf,textfont=sf]{subfig}
\usepackage{textcomp}
\usepackage{stfloats}
\usepackage[sf,bf,small,raggedright]{titlesec}
\usepackage{amsmath,amsfonts,amssymb,mathrsfs,amsthm}
\usepackage{array}
\usepackage{url}
\usepackage{verbatim}
\usepackage{graphicx}
\usepackage{cite}
\usepackage[pdfencoding=auto]{hyperref}
\usepackage{bookmark}
\usepackage{comment}
\usepackage{bbm}
\usepackage{fancyhdr}
\usepackage{amsbsy}
\usepackage{amscd}
\usepackage{latexsym}
\usepackage{pdfsync}
\usepackage{blkarray}
\usepackage{xcolor}
\usepackage{enumerate}
\usepackage{tcolorbox}
\usepackage[export]{adjustbox}
\usepackage{wrapfig}
\usepackage[parfill]{parskip}
\usepackage{bm}
\usepackage{times}

\usepackage{psfrag,epsf}
\usepackage{tikz}
\usepackage{upgreek}

\newcommand{\N}{\mathbb N}

\newcommand{\B}{\mathcal B}

\newcommand{\E}{\mathbb E}
\newcommand{\VAR}{\mathbb{V}\text{ar}}
\newcommand{\VARd}[1]{\VAR\!\left(#1\right)}

\newcommand{\F}{\mathcal F}

\newcommand{\R}{\mathbb R}
\renewcommand{\S}{\mathcal S}
\newcommand{\U}{\mathcal U}
\newcommand{\V}{\mathcal V}
\newcommand{\X}{\mathcal X}

\renewcommand{\em}{\itshape}

\newcommand{\prodm}{\mu_{\otimes}}

\newcommand{\one}{\mathbf{1}}

\newtheorem{thm}{Theorem}

\newtheorem{lem}[thm]{Lemma}

\newtheorem{rem}[thm]{Remark}

\newtheorem{cond}[thm]{Condition}

\DeclareMathAlphabet{\mathpzc}{OT1}{pzc}{m}{it}
\newcommand{\x}{\mathbf x}

\renewcommand{\hat}{\widehat}
\newcommand{\hhat}[1]{\hat{#1}^*}
\renewcommand{\tilde}{\widetilde}

\title{Estimating the Mixing Coefficients of Geometrically Ergodic Markov Processes}

\author{
  Steffen Gr\"unew\"alder\thanks{Department of Mathematics, University of York, York, UK. 
  Email: steffen.grunewalder@york.ac.uk}
  \and
  Azadeh Khaleghi\thanks{ENSAE - CREST, Institut Polytechnique de Paris, Palaiseau, France. 
  Email: azadeh.khaleghi@ensae.fr}
}

\date{} % empty date (no date printed); or use \date{\today}

\begin{document}

\maketitle

\begin{abstract}
We propose methods to estimate the $\beta$-mixing coefficients of a real-valued geometrically ergodic Markov process from a single sample-path $X_0,X_1, \dots,X_{n-1}$. Under standard smoothness conditions on the densities, namely,  that
the joint density of the pair $(X_0,X_m)$ for each $m$ lies in a Besov space $B^s_{1,\infty}(\R^2)$ for some known $s>0$, we obtain a rate of convergence of order $\mathcal{O}(\log(n) n^{-[s]/(2[s]+2)})$ for the expected error of our estimator in this case; we use $[s]$ to denote the integer part of the decomposition $s=[s]+\{s\}$ of $s \in (0,\infty)$ into  an integer term and a {\em strictly positive} remainder term $\{s\} \in (0,1]$. We complement this result with a high-probability bound on the estimation error, and further obtain analogues of these bounds in the case where the state-space is finite. Naturally no density assumptions are required in this setting; the expected error rate is shown to be of order $\mathcal O(\log(n) n^{-1/2})$. The theoretical results are complemented with empirical evaluations.
\end{abstract}

\noindent%
{\it Keywords:}  $\beta$-mixing, estimation, geometric ergodicity, Markov process, consistency

\section{\sc \bfseries Introduction}\label{sec:intro}
Temporal dependence in time-series can be quantified via various notions of {\em mixing}, which capture how events separated over time may depend on one another. The dependence between the successive observations in a stationary sequence implies that the sequence contains less {\em information} as compared to an i.i.d. sequence with the same marginal distribution. This can negatively affect the statistical guarantees for dependent samples. In fact, various mixing coefficients 
explicitly appear in the concentration inequalities involving dependent and functions of dependent sequences, making them looser than their counterparts derived for i.i.d. samples, see, e.g.
\cite{IBR62,viennet1997inequalities, RIO99,SAM00, RIO2000905,BER06, bradley2007introduction,kontorovich2008concentration,BOSQ12} 
for a non-exhaustive list of such results. Thus, in order to be able to use these inequalities in finite-time analysis, one is often required to assume known bounds on the mixing coefficients of the processes. Moreover, the quality of the assumed upper-bounds on the mixing coefficients directly translates to the strength of the statistical guarantees involving the sequence at hand. Therefore, one way to obtain strong statistical guarantees for dependent data, is to first estimate the mixing coefficients from the samples, and then plug in the estimates (as opposed to the pessimistic upper-bounds) in the appropriate concentration inequalities. Estimating the mixing coefficients can more generally lead to a better understanding of the dependence structure in the sequence. 

In this paper, we study
  the problem of estimating the $\beta$-mixing coefficients of a real-valued Markov chain from a finite sample-path, in the case where the process is stationary and geometrically ergodic. A stationary, geometrically ergodic Markov process is absolutely regular with $\beta(m)\to 0$ {\em at least exponentially fast} as $m \to 0$. 
We start by recalling the relevant concepts.

\textbf{$\beta$-mixing coefficients.} 
Let $(\Omega, \F, \mu)$  be a probability space. The $\beta$-dependence $\beta(\U,\V)$ between the $\sigma$-subalgebras $\U$ and $\V$ of $\F$ is defined as follows. Let $\iota(\omega) \mapsto (\omega,\omega)$ be the injection map from $(\Omega,\F)$ to  $(\Omega \times \Omega,\U \otimes \V)$, where $\U \otimes \V$ is the product sigma algebra generated by $\U \times \V$.
 Let $\prodm$ be the probability measure defined on  $(\Omega \times \Omega,\U \otimes \V)$ obtained as the pushforward measure of  $\mu$ under $\iota$. Let $\mu_{\U}$ and $\mu_{\V}$ denote the restrictions of $\mu$ to $\U$ and $\V$ respectively. Then
\begin{align}\label{eq:beta-dep}
\beta(\U,\V)&:=\sup_{W \in \U \otimes \V} |\prodm(W) - \mu_{\U}\times \mu_{\V}(W)|
\end{align}
where $\mu_{\U}\times \mu_{\V}$ is the product measure on $(\Omega \times \Omega,\U \otimes \V)$ obtained from $\mu_{\U}$ and $\mu_{\V}$. 
This 
leads to the sequence $\boldsymbol{\beta}:=\langle\beta(m)\rangle_{m \in \N}$ of $\beta$-mixing coefficients of a process ${\bf X}$, where 
\begin{equation*}
\beta(m):=\sup_{j \in \N}\beta(\sigma(\{X_{t}: 0\leq t \leq j-1\}),\sigma(\{X_{t}: t \geq j+m\})).
\end{equation*}
A stochastic process is said to be $\beta$-mixing  or absolutely regular if $\lim_{m \rightarrow \infty}\beta(m)=0$.

\textbf{Geometrically ergodic Markov processes.} In this paper, we are concerned with real-valued stationary Markov processes that are {\em geometrically ergodic}. Recall that a process is stationary if for every $m,\ell \in \N$ the marginal distribution on $\R^m$ of $(X_{\ell},\dots, X_{m+\ell})$ is the same as that of $(X_0,\dots, X_m)$. In the case of stationary Markov processes, the $\beta$-mixing coefficient $\beta(m),~m \in \N$ can be simplified to the $\beta$-dependence between the $\sigma$-algebras generated by  $X_0$ and $X_{m}$ respectively, i.e. 
\begin{equation*}
\beta(m) = \beta(\sigma(X_0),\sigma(X_{m}))
\end{equation*}
\cite[vol. 1 pp. 206]{bradley2007introduction}. A stationary Markov process is said to satisfy ``geometric ergodicity'' if there exists Borel functions $f:\R \to (0,\infty)$ and $c:\R\to (0,\infty)$ such that for $\rho$-a.e. $x \in \R$ and every $m \in \N$, it holds that 
\begin{equation*}
    \sup_{B \in \B(\R)}|p_m(x,B)-\rho(B)|  \leq f(x)e^{-c(x)m}
\end{equation*} where $p_m(x,B)$ defined for $x \in \mathbb R$ and $B \in \B(\R)$ is the regular conditional probability distribution of $X_m$ given $X_0$ and $\rho$ denotes the marginal distribution of $X_0$ \cite[vol. 2 Definition 21.18 pp 325]{bradley2007introduction}. It is well-known - see, e.g. \cite[vol. 2 Theorem 21.19 pp. 325]{bradley2007introduction} -  that a stationary, geometrically ergodic Markov process is absolutely regular with $\beta(m)\to 0$ {\em at least exponentially fast} as $m \to 0$. This means that in this case the  $\beta$-mixing coefficients of the process satisfy 
\begin{equation}\label{eq:betambound}
\beta(m) \leq \eta e^{-\gamma m}
\end{equation}
for some $\eta,~\gamma \in (0,\infty)$ and all $m \in \N$.

\textbf{Overview of the main results.} Our first result involves the estimation of $\beta(m)$ for each $m = 1,2,\dots$, of a real-valued geometrically ergodic Markov process from a finite sample path $X_0,X_1, \dots,X_n$. Our main assumption in this case is that the joint density $f_m$ of the pair $(X_0,X_m)$ lies in a Besov space $B^s_{1,\infty}(\R^2)$ ; roughly speaking, this implies that $f_m$ has $[s]$ many weak derivatives. As discussed above, for a geometrically ergodic Markov process, there exists some $\eta, \gamma >0$ such that $\beta(m) \leq \eta e^{-{\gamma} m}$. Given some $\eta,~\gamma \in (0,\infty)$ we show in Theorem~\ref{lem:expected_beta_error} that 
\begin{align*}
\E |\beta(m)-\hat{\beta}_N(m)| &\in \mathcal{O}
(\log(n) n^{-\frac{[s]}{2[s]+2}}),~\text{for all $m \lesssim (\log n)/\gamma$}
\end{align*} 
where, $\hat{\beta}_N$ is given by \eqref{eq:beta_m_n} with $N \approx \gamma n/\log n$ (given in Condition~\ref{ass}). 
Moreover, there exists a constant $\zeta > 0$ such that with probability $
1 - \zeta \log(n) n^{-\frac{[s]}{2[s] + 2}}$ it holds that
\[
|\beta(m)-\hat{\beta}_N(m)| \in \mathcal{O}(\log^2(n) n^{-\frac{[s]}{2[s]+2}}).
\]
The constants hidden in the $\mathcal{O}$-notation are included in the full statement of the theorem. An important observation is that neither $\eta$ nor $\gamma$ affect the rate of convergence. However, a factor $1/\gamma$ appears in the constant, and a pessimistic upper-bound on the mixing coefficients (i.e. a small $\gamma$) can lead to a large constant in the bound on the estimation error.    
Theorems~\ref{prop:beta_d} and \ref{prop:beta_dvc} are concerned with a different setting where the state-space of the Markov processes is finite. In this case, we 
do not require any assumptions on the smoothness of the densities, and the rates obtained here match that provided in Theorem~\ref{lem:expected_beta_error} if we let $s\to \infty$. 
For the estimate $\hat{\beta}_N(m)$ given by \eqref{eq:betahat_d} with $N\approx (\log n)/\gamma$ and every 
 $m \lesssim (\log n)/\gamma$ we have,
\begin{align}\label{bound_exp_d}
\E|\hat{\beta}_{N}(m)-\beta(m)|\lesssim \frac{|\X|}{\gamma}\log(n) n^{-1/2}.
\end{align}
Moreover, 
\begin{equation}
\Pr(|\hat{\beta}_{N}(m)-\beta(m)| \geq \epsilon ) \lesssim|\X| n^{-1/2} + |\X|^2\exp\left(-\frac{\gamma n\epsilon^2}{|\X|^4 \log n}\right)
\end{equation}
 for $\epsilon >0$.
We refer to the statement of Theorem~\ref{prop:beta_d} for the explicit constants. Observe that we have a factor $|\X|/\gamma$ in \eqref{bound_exp_d}. In other words, $\gamma$ has the same effect here as in the bound provided in Theorem~\ref{lem:expected_beta_error}, and the constant increases linearly in the size of the state-space.
In this setting, apart from estimating the individual mixing coefficients we can also simultaneously estimate $\beta(m)$ for $m$ up to some $k^\dagger \lesssim (\log n)/\gamma$ (see the full statement of Theorem~\ref{prop:beta_dvc} for a specification). 
The analysis relies on a VC-argument in place of a union bound which leads to tighter error bounds in this case. 
We obtain,
\begin{align*}
\E(\sup_{m \leq k^{\dagger}}|\hat{\beta}_{N}(m)-\beta(m)|)\lesssim \frac{1}{\gamma} \log (n|\X|)|\X|^2 n^{-1/2},
\end{align*}
and, 
\[\Pr(\sup_{m \leq k^{\dagger}}|\hat{\beta}_{N}(m)-\beta(m)| \geq \epsilon ) \lesssim |\X|^2 \log(n|\X|)n^{-1/2}
+|\X|^2 \exp\left(-\frac{\gamma n\epsilon^2}{|\X|^4\log n}\right)\]
for all $\epsilon >0$. Note that the parameter $k^{\dagger}$ does not explicitly appear on the right hand side of the above bounds as it has already been substituted for in the calculation. 

In our empirical evaluations provided in Section~\ref{sec:exp}, we highlight one of the key features of our estimators: their model-agnostic nature, which enables them to function effectively without prior knowledge of the underlying time-series models, requiring only assumptions on the smoothness of densities and loose bounds on the mixing rate. We specifically demonstrate that a natural model-specific $\beta$-mixing estimator, designed for {\em autoregressive AR} process, suffers performance drop in the presence of model-mismatch. In contrast to our estimator, this model-specific estimator exhibits bias and fails to consistently estimate the mixing coefficient, highlighting its vulnerability to inaccuracies when the underlying model is not correctly specified.

\textbf{Related literature.}
Research on the direct estimation of the mixing coefficients is relatively scarce. In the asymptotic regime, \cite{NOB06} used hypothesis testing to give asymptotically consistent estimates of the polynomial decay rates for covariance-based mixing conditions. \cite{KL23} proposed 
asymptotically consistent estimators of the $\alpha$-mixing and $\beta$-mixing coefficients of a stationary ergodic process from a finite sample-path. Since in general, rates of convergence are non-existent for stationary ergodic processes (see, e.g. \cite{sheilds96}), their results necessarily remain asymptotic and no rates of convergence can be obtained. 
An attempt at estimating $\beta$-mixing coefficients has also been made by \cite{MCD15}. Despite our best attempts, we have been unable to verify some of their main claims and have particular reservations about the validity of their rates. More specifically, their main  theorem (Theorem 4) suggests a rate of convergence of order $\log(n)n^{-1/2}$ for their estimator, independently of the dimension of the state-space and under the most minimal smoothness assumptions on the densities. Given that under these conditions a density estimator is known to have a dimension-dependent rate of about $n^{-1/(2+d)}$ even when the samples are iid \cite[pp. 404]{nickl}, it is highly unlikely that  a dimension-independent rate would be achievable for an estimator of the $\beta$-mixing coefficient. An interesting body of work exists for a different, yet related problem, concerning the estimation of the mixing times of finite-state Markov processes \cite{Peres19, Kont19,Wolfer20}. 
Indeed, in their recent work, \cite{wolfer2024optimistic} strategically exploit the relationship between the mixing time and the $\beta$-mixing coefficients of an ergodic Markov process with a discrete alphabet, and use this connection to construct estimators for the average mixing time of the process. Moreover, they propose an approach that reduces the dependence on the size of the state space in the discrete-alphabet setting. This in turn, enables them to handle certain countably infinite state-space chains and to estimate the corresponding mixing coefficients.

\textbf{Organization.} The remainder of the paper is organized as follows. In Section~\ref{sec:pre} we introduce some preliminary notation and basic definitions. After formulating the problem in Section~\ref{sec:prob}, we propose our estimators in Section~\ref{sec:res}, considering real-valued as well as we finite-valued sample-paths; the technical proofs are deferred to Section~\ref{sec:proofs}. In Section~\ref{sec:exp}, we present empirical evaluations of our estimators, focusing on their model-agnostic nature, to support the primary theoretical results.
%we assess the performance of our estimators within the context of the classical {\em autoregressive} (AR) time-series models. %
In Section~\ref{sec:conc} we provide concluding remarks and explore potential future research directions.

\section{\sc \bfseries Preliminaries}
\label{sec:pre}
In this section we introduce notation and provide some basic definitions. We denote the non-negative integers by $\N:=\{0,1,2,\dots\}$. If $\mathcal U$ and $\mathcal V$ are two $\sigma$-algebras, we will occasionally use standard measure-theoretic notation as follows: their product $\sigma$-algebra is denoted by 
$\mathcal U \otimes \mathcal V := \sigma(\mathcal U \times \mathcal V)$, 
and the $\sigma$-algebra generated by their union is denoted by 
$\mathcal U \vee \mathcal V := \sigma(\mathcal U \cup \mathcal V)$.
If $s\in (0,\infty)$, then we let $s=[s]+\{s\}$ be decomposed into its integer part $[s] \in \N$ and a {\em strictly positive} remainder term $\{s\} \in (0,1]$. In particular, if $s=i$ for any $i > 0 \in \N$ then, $[s]=i-1$ and $\{s\}=1$. 
As part of our analysis in Section~\ref{subsec:real}, we impose classical density  assumptions on certain finite-dimensional marginals of the Markov processes considered. The densities satisfy standard smoothness conditions as controlled by the parameters of appropriate Besov spaces.

\textbf{Besov Spaces $B_{p,\infty}^s(\R^d)$.}  
For an arbitrary function $f:\R^d \to \R$ and any vector $h \in \R^d$ let $\varDelta_h f(x):=f(x+h)-f(x)$ be the first-difference operator, and obtain higher-order differences inductively by $\varDelta_h^r f:=(\varDelta_h \circ \varDelta_h^{r-1})f$ for $r=2,3,\dots$\,. Denote by $L_p(\R^d),~p \geq 1$ the $L_p$ space of functions $f: \R^d \to \R$.
For $s>0$ the Besov space $B_{p,\infty}^s(\R^d)$ is defined as
\begin{align}\label{eq:defn:besov_space}
B_{p,\infty}^s(\R^d) := \{f \in L_p(\R^d): \|f\|_{B_{p,\infty}^s(\R^d)} < \infty\}
\end{align}
with Besov norm
$ 
\|f\|_{B_{p,\infty}^s(\R^d)}:=\|f\|_1+{{\sup_{0<t<\infty}}} t^{-s}\sup_{|h| \leq t}\|\varDelta_h^r f\|_1 
$ 
where $|v|:=\sum_{i=1}^d |v_i|$ for $v =(v_1,\dots,v_d)\in \R^d$, and $r$ is any integer such that $r > s$ \cite{bennett1988interpolation}. 
Denote by $W_p^{r}(\R^d),~r \in \N$ the Sobolev space  of functions  $f:\R^d \to \R$.
We rely on the following interpolation result concerning $W_p^r(\R^d)$ and $B_{p,\infty}^s(\R^d)$.
\begin{rem}\label{rem}
As follows from \cite[Proposition 5.1.8 and Theorem 5.4.14]{bennett1988interpolation} 
for any $r_0,~r_1 \in \N$ and 
$\varsigma := (1-\theta)r_0+\theta r_1,~\theta\in(0,1)$ 
it holds that 
\begin{align}\label{eq:interpol}
W_p^{r_0}(\R^d) \cap W_p^{r_1}(\R^d) \hookrightarrow B_{p,\infty}^{\varsigma} (\R^d)\hookrightarrow W_p^{r_0}(\R^d) + W_p^{r_1}(\R^d).
\end{align}
If $r_0 < r_1$, then the left and right hand sides of \eqref{eq:interpol} reduce to $W_p^{r_1}(\R^d)$ and $W_p^{r_0}(\R^d)$ respectively.  Furthermore, in this case we have $\|f\|_{B_{p,\infty}^{\varsigma}(\R^d)} \geq \|f\|_{W_p^{r_0}(\R^d)} $. To see this consider the $K$-functional \[K(f,t;W_p^{r_0}(\R^d),W_p^{r_1}(\R^d)):= \inf\{\|f_0\|_{W_p^{r_0}(\R^d)}+t\|f_1\|_{W_p^{r_1}(\R^d)}:f=f_0+f_1\}\] and observe that by \cite[Theorem 5.4.14 and Definition 5.1.7]{bennett1988interpolation} we have
\begin{align*}
\|f\|_{B_{p,\infty}^{\varsigma}(\R^d)} &=\sup_{t>0}t^{-\theta}K(f,t;W_p^{r_0}(\R^d),W_p^{r_1}(\R^d)) \\&\geq K(f,1;W_p^{r_0}(\R^d),W_p^{r_1}(\R^d)) \\ &\geq \inf\{\|f_0\|_{W_p^{r_0}(\R^d)}+\|f_1\|_{W_p^{r_0}(\R^d)}:f=f_0+f_1\}\\ &\geq  \|f\|_{W_p^{r_0}(\R^d)}
\end{align*}
 In particular, consider the Besov space $B_{1,\infty}^s(\R^d)$ for some $s \in (0,\infty)$ and take $\theta=\{s\}/2$, $r_0=[s]$ and $r_1=[s]+2$ for the above convex combination. Then $\|f\|_{B_{1,\infty}^s(\R^d)} \geq \|f\|_{W_{1}^{[s]}(\R^d)}$. 
\end{rem} 
\section{\sc \bfseries Problem Formulation}\label{sec:prob}
We are given a sample $X_0,X_1,\dots,X_{n}$ generated by a stationary geometrically ergodic Markov process taking values in some $\X \subseteq \mathbb R$. As discussed in Section~\ref{sec:pre} such a process is known to have a sequence of $\beta$-mixing coefficients such that 
\[
    \beta(m) \leq \eta e^{-\gamma m},~m \in \N
\]
for some unknown constants $\eta,~\gamma \in (0,\infty)$. The mixing coefficient $\beta(m)$ is unknown, and it is our objective to estimate it. In Section~\ref{sec:res} we introduce our estimators together with theoretical bounds on the estimation error. In Section~\ref{sec:exp} we demonstrate the performance of the proposed estimators on samples generated by standard time-series models. 
\section{\sc \bfseries Theoretical Results}\label{sec:res}
In this section, we present our estimators for $\beta(m)$. We consider two different settings, depending on the state-space. First, in Section~\ref{subsec:real}, we consider the case where the process is real-valued and its one and two-dimensional marginals have densities with respect to the Lebesgue measure. Next, in Section~\ref{subsec:finite}, we consider the case where the state-space $\X$ is finite. In this setting we do not require any density assumptions, and are able to control the estimation error simultaneously for multiple values of $m$. This is stated in Theorem~\ref{prop:beta_dvc}.
The following notation is used in both settings. 
For each $m \in \N$, denote by $P_m$ the joint distribution of the pair $(X_0,X_{m})$ so that for each $U \in \B(\X^2)$ we have \[\Pr(\{(X_0,X_{m}) \in U\}) = P_m(U).\] 
By stationarity, \[\Pr(\{(X_{t},X_{m+t}) \in U\}) = P_m(U),~t \in {\N}.\] 

\subsection{Real-valued state-space}\label{subsec:real} 
We start by considering the case where $\X$ is any subset of $\R$. We assume that $P_m$ has a density $f_m: \R^2 \to \R^2$ with respect to the Lebesgue measure $ \lambda_2$ on $\R^2$. Denote by $P_0$ the marginal distribution of $X_t,~t \in \N$ whose density $f_0: \R \to \R$ with respect to the Lebesgue measure $\lambda$ on $\R$ can be obtained as 
$f_0(x) = \int_{\R} f_m((x,y))d\lambda(y)$ for $x \in \R$.
It follows that 
\begin{equation*}
    \beta(m) = \frac{1}{2}\int_{\R^2}|f_m - f_0 \otimes f_0| d\lambda_2.
\end{equation*}
For some fixed $k \in m,m+1,\dots,\lfloor n/8 \rfloor$ let $N=N(k,n):=\lfloor \frac{n-k}{2(k+1)} \rfloor$ and define the sequence of tuples 
\[Z_i = ( X_{2i(k+1)},X_{2i(k+1)+m}),~i =0,1,2,\dots, N.\] 
Define the Kernel Density Estimator (KDE) of $f_{m}$ as
\begin{align}\label{eq:kde}
\hat{f}_{m,N}(z) = \frac{1}{N h_N^2}\sum_{i=0}^{N-1}K\left (\frac{z-Z_i}{h_N}\right)
\end{align}
with kernel  $K: \R^2 \to \R$ and bandwidth $h_N >0$.
Marginalizing we  obtain an  empirical estimate of $f_0$, i.e. $\hat{f}_{0,N}(x):=\int_{\R} \hat{f}_{m,N}(x,y) d\lambda(y)$. 
We define an estimator of $\beta(m)$ as
\begin{equation}\label{eq:beta_m_n}
\widehat{\beta}_{N}(m)=\frac{1}{2}\int_{\R^2}|\hat{f}_{m,N}-\hat{f}_{0,N}\otimes\hat{f}_{0,N}|d\lambda_2
\end{equation}
where $\otimes$ denotes the tensor product. Note that to simplify notation in \eqref{eq:kde} and \eqref{eq:beta_m_n}, we have omitted the dependence of $N$ on the choice $k$. An optimal value for $k$, denoted by $k^{\star}$, is provided in Condition~\ref{ass}.  
In our analysis we make standard assumptions (see, Condition~\ref{ass} below) about the smoothness of $f_m$ and the order of the kernel $K$ in \eqref{eq:kde}. Recall that a bivariate kernel is said to be of order $\ell$ if 
\begin{enumerate}
\item $c_{\ell}(K):=\displaystyle \sum_{\substack{i,j\in \N \\i+j=\ell}}\displaystyle{\int_{z=(z_1,z_2)\in\R^2} |z_1|^i |z_2|^j |K(z)|d\lambda_2(z)} <\infty$ 
\item $\displaystyle{\int_{z\in\R^2} K(z)d\lambda_2(z)} = 1$ 
 and $\displaystyle{\int_{z=(z_1,z_2)\in\R^2} z_1^i z_2^j K(z)d\lambda_2(z)} = 0$ for all $i,~j\in \N,~ i+j<\ell$
\end{enumerate}
\begin{cond}\label{ass}
The density $f_m \in B_{1,\infty}^s(\R^2)$ for some $s >1$ and $\|f_m\|_{B_{1,\infty}^s(\R^2)}\leq \varLambda$ for some $\varLambda \in (0,\infty)$. 
It is further assumed that the $\ell^{\text{th}}$ moments of the pair $(X_0,X_m)$ are finite for $\ell = 1,\dots,\lceil s \rceil$, and
$\int_{\R^2} f_m(z) (1 + \|z\|^2) d\lambda_2(z), \int_{\R^2} K^2(z) (1 + \|z\|^2) d\lambda_2(z)
< \infty$.
We assume that some $(\eta,\gamma) \in (0,\infty)^2$ is provided satisfying $\beta(m) \leq \eta e^{-\gamma m}$.
The estimator $\widehat{\beta}_N(m)$ given by \eqref{eq:beta_m_n} for $~m=1,\dots,k^{\star}$ is obtained via 
\begin{itemize}
    \item [i.] a convolution kernel $K$ of order $[s]$ such that $c_0:=\int_{\R^2}|K(z)|d\lambda_2(z)<\infty $.
    \item[ii.] and a bandwidth of the form $h_N= (c\varLambda)^{-\frac{[s]}{[s]+1}} \left(\frac{[s]-1}{2}\right)^{-\frac{[s]}{2[s]+2}} n^{-\frac{1}{2[s]+2}} $
    \end{itemize} where,
    $\displaystyle c:=\frac{c_{[s]}(K)}{[s]!}$\, and 
    $N=\lfloor \frac{n-k^{\star}}{2(k^{\star}+1)} \rfloor $ with
\begin{equation}\label{eq:k_star_defn}%%%A added \lfloor \rfloor
k^{\star}= \left \lfloor \frac{1}{\gamma}\left (\log  \frac{\gamma  \eta}{8C} + \left (\frac{3[s]+2}{2[s]+2}\right)\log n \right )\right\rfloor ,\end{equation}
 and 
\[C:=(2+c_0)(L_1)^{\frac{[s]}{[s]+1}}(c\varLambda)^{\frac{1}{[s]+1}}\] with 
$L_1^2 :=\sqrt{2} \int_{\R^2} f(z) (1 + \|z\|^2) d\lambda_2(z) \int_{\R^2} K^2(z) (1 + \|z\|^2) d\lambda_2(z) $ by choosing the hyperparameter in \cite[Remark 5.16]{nickl} to be equal to 2.
~\\
\end{cond}
We are now in a position to state our main result, namely,  Theorem~\ref{lem:expected_beta_error} below which provides bounds on the estimation error of $\widehat{\beta}_N$ given by \eqref{eq:beta_m_n}, when the assumptions stated in 
Condition~\ref{ass} are satisfied. Note that  this condition is fulfilled  by a number of standard  models. For instance, consider the stationary, geometrically ergodic $AR(1)$ model \[X_{t+1}=aX_t+\epsilon_t,~t \in \N\] where $\epsilon_t\sim \mathcal N(0,\sigma^2)$ for some $\sigma>0$ and where $X_0\sim\mathcal N(0,\sigma^2/(1-a^2))$ for some $a \in \mathbb R$ with $|a|<1$. All of the finite-dimensional marginals of this process are Gaussian; in particular, its marginal and joint densities $f_0$ and $f_m,~m=1,2,\dots$,  being infinitely differentiable, lie in any of the Besov spaces that we consider in this paper. Note that the corresponding Besov norm will be influenced by $\sigma$. In particular, for very small values of $\sigma$ the densities will have a sharp peak and the Besov norm will be large. 
\begin{thm}\label{lem:expected_beta_error} 
Under the assumptions stated and with the parameters defined in Condition~\ref{ass}, for each $m \in 1,\dots,k^{\star}$ we have
\begin{align*}
\E |\beta(m)-\hat{\beta}_N(m)| &\leq \frac{8Cn^{-\frac{[s]}{2[s]+2}}}{\gamma}\left ( e^\gamma+\log  \frac{\gamma  \eta}{8C} + \left (\frac{3[s]+2}{2[s]+2}\right )
 \log n  \right) 
\end{align*} 
Moreover, with probability at least $
1 - 
 8C
n^{-\frac{[s]}{2[s] + 2}}$ it holds that
\[
|\beta(m)-\hat{\beta}_N(m)| \leq 64 \bigl(1+ \frac{c_0}{2}\bigr) (C_1 + C_2 \log(n) +\frac{6\|K\|_1}{\gamma} \log^2(n)) n^{-\frac{[s]}{2[s]+2}}.
\]
where
\begin{align*}
C_1  &= \frac{1}{\gamma} \log\left(\frac{\gamma \eta}{8C}\right) \left(3(L_1)^{\frac{[s]}{[s]+1}}(c\varLambda)^{\frac{1}{[s]+1}} + \varLambda (c\varLambda)^{-\frac{[s]^2}{[s]+1}} \left(\frac{[s]-1}{2}\right)^{-\frac{[s]^2}{2[s]+2}}\right),     \\
C_2&=\frac{4\|K\|_1}{\gamma} \log\left(\frac{\gamma \eta}{8C}\right)  + \frac{3}{2\gamma} \left(3(L_1)^{\frac{[s]}{[s]+1}}(c\varLambda)^{\frac{1}{[s]+1}} +  \varLambda (c\varLambda)^{-\frac{[s]^2}{[s]+1}} \left(\frac{[s]-1}{2}\right)^{-\frac{[s]^2}{2[s]+2}} \right).
\end{align*}
\end{thm}
\noindent {\em See Section~\ref{sec:proofs} for a proof.}
\begin{rem}
  We refer to Section~\ref{sec:exp} for a discussion on natural choices for $k^{\star}$, $\eta$ and $\gamma$ in the context of some standard time-series models such as AR$(1)$.
\end{rem}
\subsection{Finite state-space}\label{subsec:finite}
In the special case where the state-space 
$\X$ of the Markov process is finite, we can relax the density assumptions and obtain an empirical estimate of $\beta$ by counting frequencies. 
More specifically, in this case, for each $t \in \N$ the $\sigma$-algebra $\sigma(X_t)$ is completely atomic with atoms $\{X_t=s\}, s \in \X$. Therefore, by \cite[vol. Proposition 3.21 pp. 88]{bradley2007introduction} we have
\begin{align}\label{eq:beta_d}
\beta(m) = \frac{1}{2}\sum_{u \in \X}\sum_{v \in \X}|P_m(\{(u,v)\})-P_0(\{u\})P_0(\{v\})|
\end{align}
where as before, $P_m$ and $P_0$ are the joint and the marginal distributions of $(X_0,X_m)$ and $X_0$ respectively. 
Given a sample $X_0,\dots,X_{n-1}$, we can obtain an empirical estimate of $\beta(m)$ in \eqref{eq:beta_d} as follows. Fix a lag of length $k \in 1,\dots,n$ (an optimal value for which will be specified in Proposition~\ref{prop:beta_d}).
Define the sequence of tuples 
$
Z_i = ( X_{2ki},X_{2k(i+1)})$,~$i =0,2,4,\dots, N-1$, with $N=N(k,n):=\lfloor \frac{n-k}{2(k+1)} \rfloor$. For each pair $(u,v) \in \X^2$ let 
\[
\hat{P}_{m,N}((u,v))=\frac{1}{N}\sum_{i=0}^{N-1}\mathbf{1}_{\{(u,v)\}}(Z_i).
\]
Similarly, for each $u \in \X$ we can obtain an empirical estimate of $P_0(u)$ as
\[
\hat{P}_{0,N}(u):=\frac{1}{2N}\sum_{i=0}^{2N-1}  \mathbf{1}_{\{u\}}(X_{ki})
.\]
Define
\begin{align}\label{eq:betahat_d}
\hat{\beta}_{N}(m) :=\frac{1}{2}\sum_{u \in \X}\sum_{v \in \X}|\hat{P}_{m,N}(\{(u,v)\})-\hat{P}_{0,N}(\{u\})\hat{P}_{0,N}(\{v\})|
\end{align}
%Theorem~\ref{prop:beta_d} below serves as an analogue to Theorem~\ref{lem:expected_beta_error} for the finite-alphabet case. 
\begin{thm}\label{prop:beta_d}
Consider a sample of length $n \in \N$ of a stationary, geometrically ergodic Markov process with finite state-space $\X$. Define $N(k,n)=\lfloor \frac{n-k}{2(k+1)} \rfloor,~n,k \in \N$. 
Let $k^{\star}:=\left \lfloor \frac{1}{\gamma}\log \left(\frac{\eta\gamma n^{3/2}}{\sqrt{2}|\X| }\right)\right\rfloor $ and $N=N(k^{\star},n)$ with $n \geq \max \left \{\left (\frac{\sqrt{2}|\X| e^{2\gamma}}{\eta \gamma}\right)^{2/3}, \left(\frac{\eta \gamma}{|\X|}\right)^{2/3} \right\}$. 

For every $m=1,\dots,k^{\star}$ and $\hat{\beta}_N(m)$ given by \eqref{eq:betahat_d} we have,
\begin{align*}
\E|\hat{\beta}_{N}(m)-\beta(m)|
&\leq \frac{\sqrt{8}|\X| n^{-1/2}}{\gamma}\left(e+\log \left(\frac{\eta\gamma}{\sqrt{2}|\X| }\right)+\frac{3}{2}\log n\right)
\end{align*}
Moreover, for $\epsilon >0$ it holds that
\begin{align*}
\Pr(|\hat{\beta}_{N}(m)-\beta(m)| \geq \epsilon ) 
%&\leq \frac{\sqrt{8}|\X| n^{-1/2}}{\log(\frac{\eta \gamma n^{3/2}}{\sqrt{8}|\X|}) }+4|\X|^2\exp\{-\frac{\gamma n\epsilon^2}{12 |\X|^4 \log n}\}\\
&\leq \frac{e\sqrt{2}|\X| n^{-1/2}}{2\log(\frac{\eta \gamma n^{3/2}}{\sqrt{2}|\X|}) }+4|\X|^2\exp\left \{-\frac{\gamma n\epsilon^2}{48 |\X|^4 \log n}\right \}
\end{align*}
\end{thm}
\noindent{\em The proof is provided in Section~\ref{sec:proofs}}. 

In this setting, we are also able to simultaneously control the estimation error for all $m=1,\dots,k^{\dagger}$ where $k^{\dagger}$ is specified in the statement of Theorem~\ref{prop:beta_dvc}. The proof relies on a VC argument which helps  replace a factor of $k^{\dagger}$ (which would have otherwise been deduced from a union bound) with a factor of $\log k^{\dagger}$. 
\begin{thm}\label{prop:beta_dvc}
Consider a sample of length $n \in \N$ of a stationary, geometrically ergodic Markov process with finite state-space $\X$. Define $N(k,n)=\lfloor \frac{n-k}{2(k+1)} \rfloor,~n,k \in \N$. 
Let $N=N(k^{\dagger},n)$ with $k^{\dagger}:=\left \lfloor \frac{1}{\gamma}\log\left( \frac{\eta \gamma n^{3/2}}{4\sqrt{2}|\X|^2\log_2(n|\X|)}\right)\right \rfloor $ and $n \geq \max \left \{\frac{4\sqrt{2}|\X|^3 e^\gamma}{\eta \gamma}, \frac{2k(k+1)}{3k+2}\right \}$. 
For $\hat{\beta}_N$ given by \eqref{eq:betahat_d} we have,
{\small \begin{align*}
\E[\sup_{m \in 1,\dots,k^{\dagger}}|\hat{\beta}_{N}(m)-\beta(m)|]\leq \frac{4\sqrt{2}|\X|^2 n^{-1/2} \log_2 (n|\X|)}{\gamma}\left(e+\log\left(\frac{\eta \gamma n^{3/2}}{4\sqrt{2}|\X|^2\log_2(n|\X|)}\right)\right)
\end{align*}}
Furthermore, for $\epsilon >0$ the probability $\Pr(\sup_{m \in 1,\dots,k^{\dagger}}|\hat{\beta}_{N}(m)-\beta(m)| \geq \epsilon )$  is at most
{\small \begin{align*}
\frac{4\sqrt{2}|\X|^2 \log(n|\X|)n^{-1/2}}{\log(\frac{\eta \gamma n^{3/2}}{8\sqrt{2}|\X|^2})}
+16|\X|^2 \log\left(\frac{3|\X|}{2\gamma}\log(\eta^{2/3} \gamma^{2/3} n)\right)\exp\left\{-\frac{\gamma n\epsilon^2}{3072|\X|^4\log n}\right\}
\end{align*}}
{\em The proof is provided in Section~\ref{sec:proofs}}. 
\end{thm}
\section{\sc \bfseries Empirical Evaluations}\label{sec:exp}
In this section, we evaluate the performance of our estimators using synthetic data generated from the well-known \textit{autoregressive (AR)} time-series model of order 1 (see, e.g., \cite{SHUM25}). Additionally, we examine the case where the autoregressive process deviates from the traditional AR model, featuring a non-Gaussian noise process. This latter setting serves to highlight the model-agnostic nature of our estimators.

A first-order autoregressive process, AR(1) can be defined as follows:
\[
X_t = \phi_1 X_{t-1} + \epsilon_t,\quad t \geq 1
\]
with parameter $\phi_1 \in \mathbb{R}\setminus \{0\}$ and an i.i.d. sequence of zero-mean normally distributed noise random variables $\epsilon_1, \epsilon_2, \ldots$ with variance $\sigma^2$. The process is stationary if $|\phi_1| < 1$ and, and by letting $X_0$ have distribution $\mathcal{N}(0,\sigma^2/(1-\phi_1^2))$ we can guarantee stationarity of the process $X_0,X_1,X_2,\ldots$.
To be able to calculate the estimation error, we compute the true $\beta$-mixing coefficients numerically. Furthermore, in the context of AR models, it is common to estimate the \textit{autocorrelation function (ACF)} which is a measure of linear (as opposed to full) dependence. More specifically, the ACF at lag $m$ is given by $\rho_m = E(X_0 X_{m})$. A standard estimator of $\rho_m$ when given a sample of length $n$ can be obtained as 
\begin{equation}\label{eq:rhohat}
\hat \rho_m = \frac{1}{(n-m)\hat \sigma^2} \sum_{t=0}^{n-m} X_t X_{t+m}, \quad \text{ with \enspace} \hat \sigma^2 = \frac{1}{n} \sum_{t=0}^{n-1} X_t^2.  
\end{equation}
In the case of an AR(1) process with Gaussian noise, an estimate of the $\beta$-mixing coefficients of the process can be readily obtained by using the ACF estimator given by \eqref{eq:rhohat}. With this observation, we obtain a (model-specific) estimate of the $\beta$-mixing coefficient of the AR(1) process. This estimator is used as a point of comparison against our own estimators of $\beta$ in this experiment. It is important to note that unlike the $\hat{\rho}_m$-based estimator, our estimators do not have prior knowledge of the underlying time-series model, beyond assumptions on the smoothness of the densities. In fact, as part of our experiments, we further emphasize the model-agnostic nature of our estimator by comparing it with the ACF-based estimator under model-mismatch conditions: when the noise process is no longer Gaussian. In this setting, while the ACF-based estimator demonstrates bias and fails to consistently estimate the mixing coefficients, our estimator maintains remains consistent.

As discussed in Section~\ref{sec:res}, our estimators rely on KDEs to estimate $\beta$ and as such require smoothness conditions on the the densities. Additionally, proper selection of $k^{\star}$, necessitates (loose) bounds on the mixing rate parameters of the process. Below, we discuss these conditions for AR(1) and provide insights into our approach for choosing $k^{\star}$ in our experiments.

\textbf{KDE estimators for AR(1).}
Classical time-series models often assume Gaussian increments. This implies in various contexts that the joint densities of $X_i$ and $X_{i+m}$ is Gaussian. Gaussian densities are infinitely often differentiable and correspond to the smoothness parameter $s$ of a Besov space approaching infinity. This yields a very fast rate of convergence of the KDE when the bandwidth $h_n$ of the KDE is adapted to this context. In particular, the bandwidth $h_n$ depends logarithmically on $n$
 and the rate of convergence will approach $\sqrt{\log(n)/n}$ in the $L^2$-distance and $\log(n)/n$ for the mean squared error independently of the dimension. 
A natural alternative to selecting the bandwidth in this way is to make use of popular heuristics for choosing $h_n$. In practice, it is common to use either \textit{Scott's rule} or \textit{Silverman's rule} to select $h_n$;  in our experiments $h_n$ is  selected according to Scott's rule, which for an $\mathbb{R}^d$-valued sample sets $h_n = n^{-1/(4+d)}$. In the context of Besov spaces this corresponds to the case $s=2$ \cite[p.404]{nickl} with corresponding rate of convergence of the KDE of order $n^{-2/(4+d)}$. Using Scott's rule leads to suboptimal rates of convergence of the KDE if our model assumptions are, in fact, correct since Scott's rule is pessimistic about the smoothness of the density and does not exploit the smoothness of the Gaussian distribution. 

\textbf{Choice of $k^{\star}$.} Another hyperparameter we are required to select is $k^{\star}$, which corresponds the gap between time-points at which we collect observations. Larger values of $k^{\star}$ lead to observations that are closer to independent. The downside of large values of $k^{\star}$ is that significant amount of information is discarded. In Condition \ref{ass}, we provided a general way to select $k^{\star}$ based on loose bounds on the mixing rate parameter. In the context of classical time-series models we have natural alternatives to this choice. In particular, for AR(1) models it is possible to bound the mixing-rate in terms of the AR(1) parameter and it is natural to replace our general assumption in Condition \ref{ass} by an AR-specific assumption on the parameters (e.g. the AR(1) parameter is smaller than $0.95$). We develop these alternative approaches to select $k^{\star}$ as part of the model specific sections that follow.

As follows from \cite[Thm.10.14]{JAN97}, for an AR(1) process we can bound $\beta(m)$ from above and below using $\rho_m$ as 
\begin{equation} \label{bnd:beta_rho}
\frac{1}{\pi} |\rho_m| - \frac{\rho_m^2 + (1/4) \rho_m^4}{(1-\rho_m)^2} \leq \beta(m) \leq \frac{1}{\sqrt{2\pi}} |\rho_m| + \frac{\rho_m^2 + (1/16) \rho_m^4}{(1-\rho_m)^2}.
\end{equation}
Note that the bound becomes tighter for small $\rho_m$ and looser as $\rho_m$ approaches $1$. For larger values of $\rho_m$ it indeed becomes trivial. The bound given in \eqref{bnd:beta_rho} allows us to control $\beta(m)$  by means of the correlation coefficients of the AR model. The correlation coefficients are in turn directly related to $\phi_1$. It is natural to work in the AR(1) context with loose assumptions on $\phi_1$ and the above discussion provides a path for turning such assumptions into bounds on $\beta(m)$ which we can then use to tune $k^{\star}$. The relationship between $\rho_m$, $m\geq 1$, and $\phi_1$ is given by  
\[\rho_m = \phi_1^m E(X_0^2)/E(X_0^2) = \phi_1^m.\] 
In particular, if we assume an upper bound $b$ on $|\phi_1|$ then $|\rho_m| \leq b^m$. Substituting this into Equation
\eqref{bnd:beta_rho} and ignoring second order terms provides us with the bound
\begin{equation}\label{eq:bound_bm}
\beta(m) \lesssim \frac{1}{\sqrt{2\pi}} b^m.
\end{equation}
Taking into account \eqref{eq:bound_bm} when optimizing the bound on the expected error given by \eqref{eq:bound}, and noting the smoothness of the Gaussian density, we propose choosing $k^{\star}$ in this case as
\begin{equation}\label{eq:kstar_AR}
\left\lfloor\frac{\log\log(1/b)+(3/2)
\log n }{\log (1/b)} \right\rfloor.
\end{equation}
Observe that for a fixed $n$, when $b \to 1$ we obtain $k^{\star} \to \infty$. Another alternative is to relax the requirement for the estimator to be based on (near)-independent samples by setting $k^{\star} = 0$. In this case, one can leverage the ergodicity of the underlying process to achieve consistency, as demonstrated in \cite{KL23}. An advantage of this approach is that it eliminates the need for bounds on the mixing rate to select $k^{\star}$, allowing us to utilize the entire dataset without discarding any data. However, this might result in a higher variance, which would need to be appropriately controlled. Although theoretical results for finite-time analysis in the context of Markov processes and KDE-based estimators are currently lacking, we conjecture that it may be possible to prove the consistency of our estimator for $k^{\star} = 0$. We provide some simulation results below which lend support to this hypothesis.

\begin{figure}[!ht]
 \begin{minipage}{0.5\textwidth}
        \centering
        \includegraphics[height=6cm]{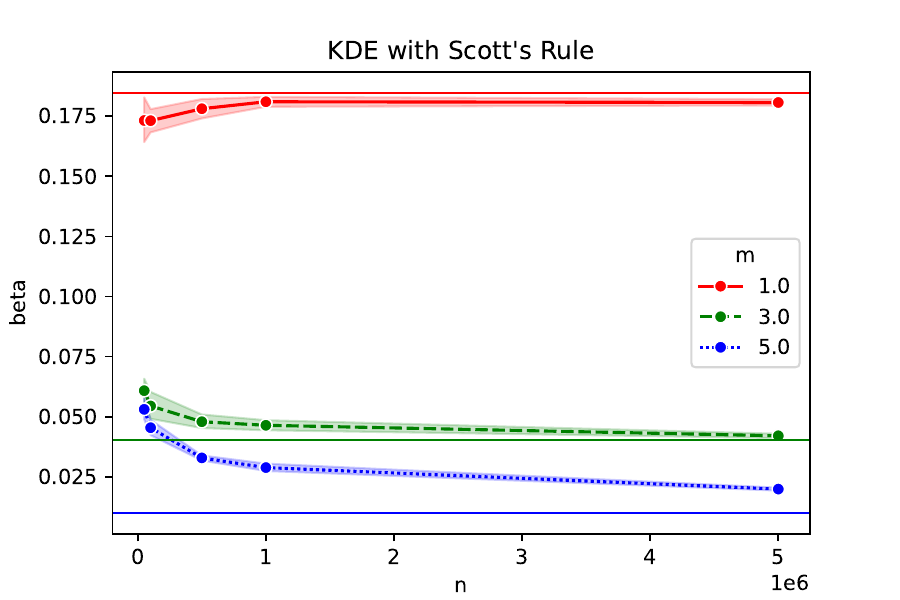}
    \end{minipage}\hfill
    \begin{minipage}{0.5\textwidth}
        \centering
        \includegraphics[height=6cm]{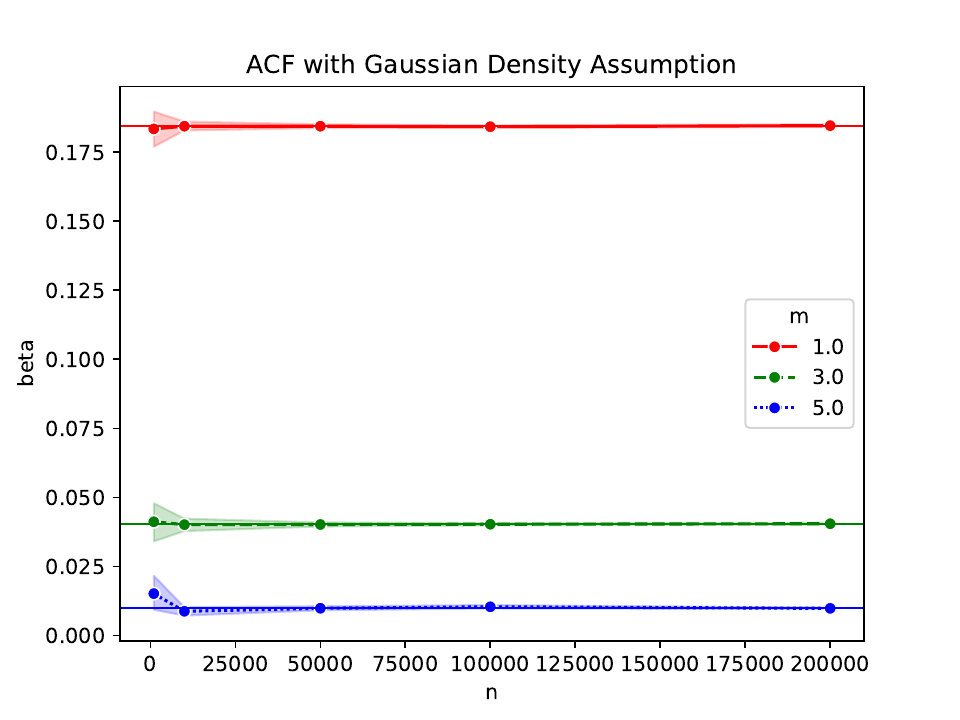}
    \end{minipage}
 \caption{The plot on the left illustrates the performance of our KDE-based estimators, while the plot on the right demonstrates the performance of the ACF-based estimators for the estimation of $\beta(m)$, $m = 1,3,5$ when samples are generated according to AR(1). The horizontal lines represent the true values of $\beta$, with the top solid line (red) corresponding to $m=1$, the middle solid line (green) to $m=3$, and the bottom solid line (blue) to $m=5$. The mean estimates (averaged over 20 rounds) are plotted for each $m$ value, and the shaded areas indicate the uncertainty in these estimates. As $n$ increases, the estimates clearly converge to their respective true $\beta$ values.
In the right plot, the results for the ACF-based estimator are shown. The ACF-based estimator takes advantage of the knowledge that the process is AR(1) and, more specifically, assumes Gaussian densities. This assumption leads to faster convergence in this particular setting, as the model assumption is accurate.
}
\label{fig:AR1}
\end{figure}
\textbf{Simulation Results.} We start by considering an AR(1) model with $\phi_1 = 1/2$ and  $\sigma^2 = 1$. 
Figure \ref{fig:AR1} shows the estimation results for the values of $\beta(m),~m=1,3$ and $5$. The plot on the left shows the results for our estimator when $k^{\star}$ is chosen according to Equation \eqref{eq:kstar_AR} and using the assumption that $|\phi_1| \leq 0.9$. The plot on the right shows the results for a naive plug-in estimator that relies on the empirical ACF to calculate the empirical joint and marginal densities of the Gaussian distribution, and integrate the absolute  difference to estimate $\beta$. In this experiment, the true $\beta$ values (numerically determined)  are $\beta(1) \approx  0.1846, \beta(3) \approx 0.0402 $ and  $\beta(5) \approx 0.00996$. 
As expected, the ACF-based estimator which relies on the fact that the process is AR(1), is able to efficiently estimate $\beta$ in this setting by calculating a simple $2 \times 2$ covariance matrix.
\begin{figure}[!ht]
\label{fig:AR1_Second}
\centering
 \begin{minipage}{0.5\textwidth}
        \centering
        \begin{tikzpicture}
        \node[inner sep=0pt]  at (0,0) {\includegraphics[width=1.1\textwidth]{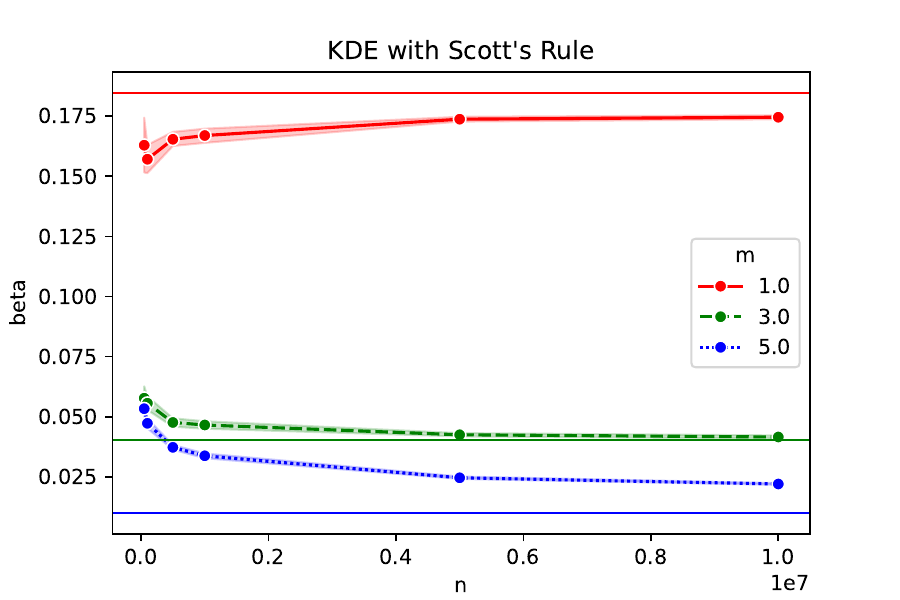}};
        \node[text width = 4cm] at (0,0) {$k^{\star}$ given by \eqref{eq:kstar_AR}};
        \end{tikzpicture}
\end{minipage}\hfill
 \begin{minipage}{0.5\textwidth}
        \centering
        \begin{tikzpicture}
        \node[inner sep=0pt]  at (0,0) {
        \includegraphics[width=1.1\textwidth]{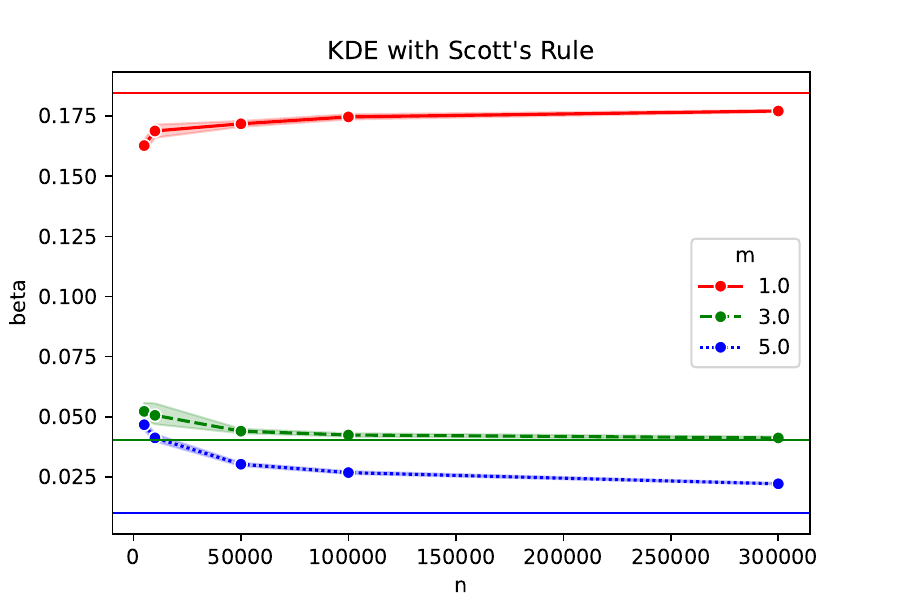}};
        \node at (0,0) {$k^{\star} = 0$};
        \end{tikzpicture}        
    \end{minipage}\\
    \begin{minipage}{0.5\textwidth}
        \centering
        \includegraphics[width=1.1\textwidth]{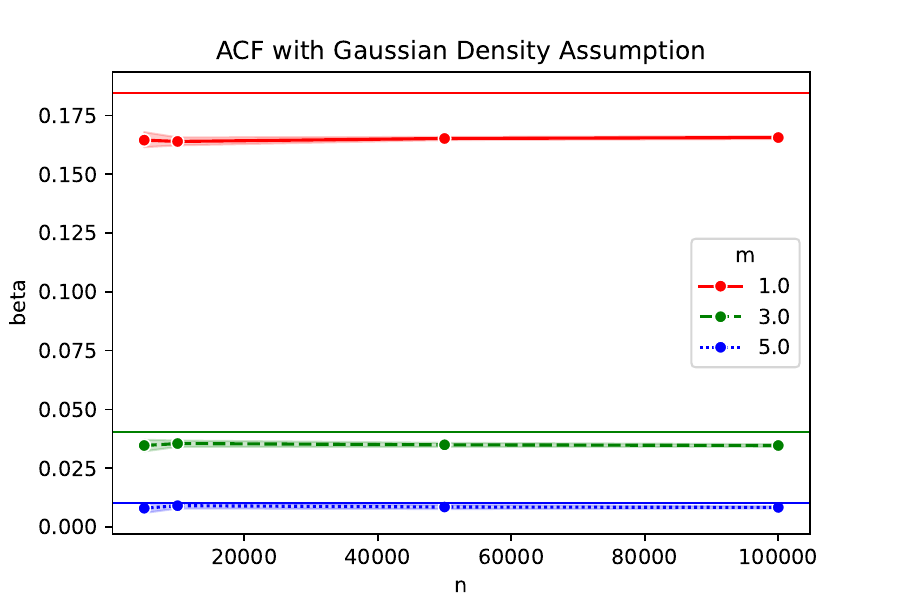}
    \end{minipage}
\caption{Performance of our estimators (top) and the ACF-based estimator (bottom) when the sample is a transformation of an AR(1) process.}
\label{fig:AR1_misspec}
\end{figure}

In a second set of experiments, we study the performance of the estimators under a mild model miss-specification. In this setting, the samples are given by $\exp(X_t) - E(\exp(X_t))$ where $X_t$ is an AR(1) process. This yields a sequence of (centered) log-normal random variables. Observe that this transformation of $X_t$ does not change the $\beta$-dependence of the original process. The results are shown in Figure \ref{fig:AR1_misspec}. The plots on the top left and right show the results of our estimator where $k^{\star}$ is chosen according to \eqref{eq:kstar_AR} and $k^{\star}=0$ respectively; the latter is more sample-efficient. Two notable observations can be gleaned from the results. 
1. The model-mismatch encountered by the ACF-based estimator in this case results in a noticeable systematic bias in the estimates, even under a minor model mis-specification. This bias is particularly prominent for $m=1$, where the dependencies are more significant. 2. Interestingly, in the case where $k^{\star}=0$, the estimator performs well, despite the samples being dependent. This leads us to believe that it may be possible to prove consistency of our KDE-estimators without the need to rely on blocks that are $k^{\star}$ steps apart.

\section{\sc \bfseries Discussion}\label{sec:conc}
We have introduced novel methods for estimating the $\beta$-mixing coefficients of a real-valued geometrically ergodic Markov process from a sample-path of length $n$. Under standard smoothness conditions on the densities, we have established a convergence rate of $\mathcal{O}(\log(n) n^{-[s]/(2[s]+2)})$ for the expected error of our estimator, and have provided a high-probability bound on the estimation error. Furthermore, we have derived analogous bounds in the case of a finite state-space, where no density assumptions are required, and demonstrate an expected error rate of $\mathcal O(\log(n) n^{-1/2})$. Although our work is primarily theoretical, we have also presented empirical evaluations to further validate the performance of the proposed estimators. In our experiments, we demonstrate the efficacy of our estimators across a range of geometircally ergodic Markov processes without being tied to any particular modeling specification, such as linear or autoregressive structures.

Our estimators rely on (near)-independent samples obtained from blocks separated by some $k^{\star}$ time-steps in the sample path. While this is a classical technique, based on Berbee's coupling lemma, allowing to control dependencies, it has its downsides. First, to determine an appropriate value for $k^{\star}$ we require bounds, albeit loose, on the mixing parameters. Moreover, using blocks that are $k^{\star}$ steps apart leads to sample-inefficiency. 
An alternative strategy could involve eliminating the need for (near)-independent samples by setting $k^{\star} = 0$. As in \cite{KL23} it may be possible to exploit the underlying process's ergodicity to ensure consistency. However, this might result in a higher variance to be controlled. While to the best of our knowledge there are currently no theoretical results available for finite-time analysis in the context of Markov processes and KDE-based estimators, we conjecture that it could be feasible to demonstrate the consistency of our estimator under the condition that $k^{\star} = 0$. We leave this investigation for future work. 

\section{\sc \bfseries Technical Proofs}\label{sec:proofs}
\setcounter{page}{1}
In this section, we provide proofs for our theorems. A common ingredient is a coupling argument for time-series, which allows one to move from dependent samples to independent blocks. This is facilitated by Lemma~\ref{lem:ourberbee} below, which is a standard result based on, commonly used in the analysis of dependent time-series, see e.g. \cite{levental1988uniform, YU1994, arcones1994central}. For completeness, we provide a proof of this lemma, which in turns relies on a coupling Lemma of \cite{berbee1979random} stated below. 
\begin{lem}\label{lem:ourberbee}
Let $X_i,~i \in \N$ be a stationary sequence of random variables with $\beta$-mixing coefficients $\beta(j),~j \in \N$. For a fixed $k,~\ell \in \mathbb N$ let $Y_i = X_{i(k+\ell)},\dots,X_{ik+(i+1)\ell}$ for $i \in {\N}$. There exists a sequence of independent random variables $Y^*_i,~i \in \N$ taking values in $\R^{\ell}$ and have the same distribution as $Y_i$ such that for every $i \in {\N}$ we have,
\[\Pr(Y^*_i \neq Y_i)  \leq \beta(k).\]
\end{lem}

\begin{lem}[\cite{berbee1979random}]\label{lem:berbee}
 Let $X$ and $Y$ be two random variables taking values in Borel spaces $\S_1$ and $\S_2$ respectively. Denote by $U$ a random variable uniformly distributed over $[0,1]$, which is independent of $(X,Y)$. There exists a random variable $Y^* = g(X,Y,U)$ where $g: \S_1\times \S_2 \times [0,1] \to \S_2$ such that 
$Y^*$ is independent of $X$ and has the same distribution as $Y$, and that $\Pr(Y^* \neq Y) = \beta(\sigma(X),\sigma(Y))$. 
\end{lem}
\begin{proof}[Proof of Lemma~\ref{lem:ourberbee}]
Let $U_j,~j \in \N$ be a sequence of i.i.d. random variables uniformly distributed over $[0,1]$ such that each $U_j$ is independent of $\sigma(\{Y_i: i \in {\N}\})$.
Set $Y^*_0=Y_0$. 
By Lemma~\ref{lem:berbee} there exists a random variable $Y^*_1 = g_1(Y^*_0,Y_1,U_1)$ where $g_1$ is a measurable function from $\R^\ell \times \R^\ell \times [0,1]$ to $\R^\ell$
such that $Y^*_1$ is independent of $Y^*_0$, has the same distribution as $Y_1$ and 
$\Pr(Y^*_1 \neq Y_1) = \beta(\sigma(Y^*_0),\sigma(Y_1))$. 
Similarly, 
there exists a random variable 
$Y^*_2 = g_2((Y^*_0,Y^*_1),Y_2,U_2)$ where $g_2$ is a measurable function from $(\R^{\ell})^2 \times \R^{\ell} \times [0,1]$ to $ \R^{\ell}$ such that $Y^*_2$ is independent of $(Y^*_0,Y^*_1)$, has the same distribution as $Y_2$ and 
$
    \Pr(Y^*_2 \neq Y_2) = \beta(\sigma(Y^*_0,Y^*_1),\sigma(Y_2))
$.
Continuing inductively in this way, at each step $j=3,4,\dots$, by Lemma~\ref{lem:berbee}, there exists a random variable 
$Y^*_j = g_j((Y^*_0,Y^*_1,\dots,Y^*_{j-1}),Y_j,U_j)$ where $g_j$ is a measurable function from $(\R^{\ell})^j \times \R^{\ell} \times [0,1]$ to $ \R^{\ell}$ such that $Y^*_j$ is independent of $(Y^*_0,Y^*_1,\dots,Y^*_{j-1})$, has the same distribution as $Y_j$ and that 
$ 
    \Pr(Y^*_j \neq Y_j) = \beta(\sigma(Y^*_0,Y^*_1,\dots,Y^*_{j-1}),\sigma(Y_j)).
$ 
It remains to show that 
$
\beta(\sigma(Y^*_0,Y^*_1,\dots,Y^*_{j-1}),\sigma(Y_j))\leq \beta(k)
$ for all $j \in \N $.
To see this, first note that $Y^*_0 = Y_0$ by definition, and that for each $i \in \N$, it holds that $Y^*_i \in \sigma((Y_0^*,Y^*_1,\dots,Y^*_{i-1}),Y_i,U_i)$, we have
\begin{align}\label{eq:biggersigma}
\sigma(Y_0^*,Y^*_1,\dots,Y^*_{j-1}) &\subseteq \U_j \vee \V_j 
\end{align}
where $\U_j = \sigma(U_1,\dots,U_{j-1})$ and 
$\V_j = \sigma(Y_0,Y_1,\dots,Y_{j-1})$.
Take any $U \in \U_j$ and $W\in \sigma(Y_j)$. We almost surely have,
\begin{align*}
    P(U \cap W| \V_j) 
    &= \E\left ( \one_U \one_W|\V_j\right )\\
    &= \E\left (\E\left (\one_U \one_W|\V_j \vee \sigma(Y_j)\right )|\V_j \right) & \text{since $\V_j \subseteq \sigma(Y_j) \vee \V_j$}\\
    & = \E\left (\one_W\E\left (\one_U |\V_j \vee \sigma(Y_j)\right )|\V_j \right) & \text{since $W \in \sigma(Y_j)$}\\
    & = \E\left (\one_W\E\left (\one_U \right)|\V_j \right) & \text{since $\U_j$ is independent of $\sigma(Y_0,\dots,Y_j)$}\\
    &=P(U)P(W|\V_j) \\
    &=P(U|\V_j)P(W|\V_j) &\text{since $\U_j$ is independent of $\V_j$}
\end{align*}
Therefore,  and $\U_j$,~$\V_j$ and $\sigma(Y_j)$ form a Markov triplet in the sense of \cite[Vol. 1 Definition 7.1 pp. 205]{bradley2007introduction}. Thus, as follows from \cite[Vol. 1 Theorem 7.2 pp. 205]{bradley2007introduction} we obtain,
\begin{equation}\label{eq:upper-bound-beta}
    \beta(\U_j \vee \V_j, \sigma(Y_j)) = \beta(\V_j, \sigma(Y_j)).
\end{equation}
In light of \eqref{eq:biggersigma} and \eqref{eq:upper-bound-beta}, 
and noting that by construction $\beta(\V_j, \sigma(Y_j))\leq \beta(k)$ we obtain
\begin{align*}
\beta(\sigma(Y^*_0,Y^*_1,\dots,Y^*_{j-1}),\sigma(Y_j)) \leq \beta(\U_j \vee \V_j, \sigma(Y_j)) \leq \beta(\V_j, \sigma(Y_j)) \leq \beta(k).
\end{align*}
\end{proof}

\begin{proof}[Proof of Theorem~\ref{lem:expected_beta_error}]
Given the sample $X_0,\dots,X_n$, consider the sequence \[Z_i = ( X_{2i(k+1)},X_{2i(k+1)+m})\] with $i =0,1,2,\dots, N-1$ where $N=N(k,n):=\lfloor \frac{n-k}{2(k+1)} \rfloor $ for some fixed $k \in m+1,\dots,\lfloor n/8 \rfloor$. {\em As part of the proof, we propose an optimal choice for $k$, see \eqref{eq:kstar}.}
Enlarge $\Omega$ if necessary in order for Lemma~\ref{lem:ourberbee} to be applicable. 
As follows from Lemma~\ref{lem:ourberbee} there exists a sequence of independent random variables $Z_i^*,~i=0,1,\dots,N-1$ each of which takes value in $\R^2$ and has the same distribution as $Z_i,~i=0,1,\dots,N-1$, with the additional property that
\begin{equation}\label{eq:ourb}
    \Pr\left (\{\exists i \in 0,\dots,N-1: Z_i^* \neq Z_i\}\right) \leq N\beta(k)
\end{equation}
Define the KDE of $f_m$ through $Z_i^*,~i=0,\dots,N-1$
\[
\hhat{f}_{m,N}(z) = \frac{1}{N h_N^2}\sum_{i=1}^{N}K\left (\frac{z-Z^*_i}{h_N}\right)
\]
with the same kernel  $K: \R^2 \to \R$ and bandwidth $h_N >0$ as in \eqref{eq:kde} and let
\begin{equation}\label{eq:beta_m_n_star}
\hhat{\beta}_{N}(m)=\frac{1}{2}\int_{\R^2}|\hhat{f}_{m,N}-\hhat{f}_{0,N}\otimes\hhat{f}_{0,N}|d\lambda_2
\end{equation}
where $\hhat{f}_{0,N}(x):=\int_{\R} \hhat{f}_{m,N}(x,y) d\lambda(y)$.

\begin{align}
    &\|f_0\otimes f_0-\hhat{f}_{0,N}\otimes \hhat{f}_{0,N}\|_1 \nonumber \\
    &= 
    \int_{x}\int_{y}|f_0(x)f_0(y)-\hhat{f}_{0,N}(x)\hhat{f}_{0,N}(y)|d\lambda(x)d\lambda(y) \nonumber \\
    &=\int_{x}\int_{y}|f_0(x)f_0(y)-\hhat{f}_{0,N}(x)f_0(y)+\hhat{f}_{0,N}(x)f_0(y)-\hhat{f}_{0,N}(x)\hhat{f}_{0,N}(y)|d\lambda(x)d\lambda(y)\nonumber \\
    &\leq \int_{y}f_0(y)\int_{x}|f_0(x)-\hhat{f}_{0,N}(x)| d\lambda(x) d\lambda(y) + \int_{x}|\hhat{f}_{0,N}(x)|\int_{y}|f_0(y)-\hhat{f}_{0,N}(y)| d\lambda(y) d\lambda(x) \nonumber\\
    &\leq (1+c_0)\|f_0-\hhat{f}_{0,N}\|_1  \label{eq:f0f0}
\end{align}
where $c_0=\int_{\R^2}|K(z)|d\lambda_2(z)$ as specified in the theorem statement. It follows that 
\begin{align}
 |\beta(m)-\hhat{\beta}_N(m)| &= 
 \frac{1}{2}\left |\int |f_m - f_0\otimes f_0| d\lambda_2 - \int |\hhat{f}_{m,N} - \hhat{f}_{0,N}\otimes \hhat{f}_{0,N}| d\lambda_2 \right|\nonumber \\
 %&\leq \frac{1}{2}\left |\int \left (|f_m - \hhat{f}_{m,N}| +|f_0\otimes f_0-\hhat{f}_{0,N}\otimes \hhat{f}_{0,N}|\right )d\lambda_2 \right|\nonumber \\
 & \leq \frac{1}{2} \int |f_m - \hhat{f}_{m,N}| d\lambda_2 +\frac{1}{2}\int |f_0\otimes f_0-\hhat{f}_{0,N}\otimes \hhat{f}_{0,N}|d\lambda_2 \nonumber \\
 &=\frac{1}{2} \|f_m - \hhat{f}_{m,N}\|_1 + \frac{1}{2}\|f_0\otimes f_0-\hhat{f}_{0,N}\otimes \hhat{f}_{0,N}\|_1 \nonumber \\
 &\leq \frac{1}{2} \left(\|f_m - \hhat{f}_{m,N}\|_1+(1+c_0)\|f_0-\hhat{f}_{0,N}\|_1\right) \label{eq:diffs}
\end{align}
where \eqref{eq:diffs} follows from \eqref{eq:f0f0}.
Next, it is straightforward to check that if $f_m \in B_{1\infty}^s(\R^2)$ with $\|f\|_{B_{1\infty}^s(\R^2)} \leq \varLambda$, then $f_0 = \int_\R f_m d\lambda \in B_{1\infty}^s(\R)$ with $\|f\|_{B_{1\infty}^s(\R)} \leq \varLambda$. 
Moreover,  observe that, as follows from Remark~\ref{rem}, for all $f \in B_{1\infty}^s(\R)$ and $g \in B_{1\infty}^s(\R^2)$ we have  $\|f\|_{W_{1}^{[s]}(\R)}\leq \|f\|_{B_{1,\infty}^{s}(\R)}$ and $\|g\|_{W_{1}^{[s]}(\R^2)} \leq \|g\|_{B_{1,\infty}^{s}(\R^2)}$.
Therefore, with the 
choice of bandwidth $h$ specified in the theorem statement, 
by \cite[Proposition 4.1.5 and Proposition 4.3.33]{nickl} and an argument analogous to that of \cite[Proposition 5.1.7]{nickl}, we obtain,
\begin{align}\label{eq:norm1}
\sup_{f_0: \|f\|_{B_{1\infty}^s(\R)} \leq \varLambda} \E \|\hhat{f}_{0,N}-f_0\|_1 \leq \widetilde{C} N^{-\frac{[s]}{2[s]+2}}
\end{align}
where 
$\widetilde{C} = 2(L_1)^{\frac{[s]}{[s]+1}}(c\varLambda)^{\frac{1}{[s]+1}}$.
Similarly, by \cite[pp. 404]{nickl} we have,
\begin{align}\label{eq:norm2}
\sup_{f_m: \|f\|_{B_{1\infty}^s(\R^2)} \leq \varLambda} \E \|\hhat{f}_{m,N}-f_m\|_1 \leq \widetilde{C} N^{-\frac{[s]}{2[s]+2}}
\end{align}
Set $C:=(2+c_0)\widetilde{C}/2$. 
Define the event $E:=\{Z_i^*= Z_i,~i \in 0,\dots,N-1\}$. 
We obtain,
\begin{align}
\E |\beta(m)-\hat{\beta}_N(m)| 
&\leq \E |\beta(m)-\hhat{\beta}_N(m) |+\E [|\hhat{\beta}_N(m) - \hat{\beta}_N(m)|\,| E^c]\Pr(E^c) \label{eq:beta_exp_diff00} \\
&\leq\E |\beta(m)-\hhat{\beta}_N(m) | + 2 N \beta(k)  \label{eq:beta_exp_diff0}\\
&\leq C N^{-\frac{[s]}{2[s]+2}}+2N\beta(k) \label{eq:beta_exp_diff1} \\
&\leq C N^{-\frac{[s]}{2[s]+2}} +2N\eta e^{-\gamma k} \label{eq:beta_exp_diff2}\\
&\leq C \left (\frac{4k}{n-4k}\right ) ^{\frac{[s]}{2[s]+2}} +\frac{n}{k}\eta e^{-\gamma k} \label{eq:beta_exp_diff3}\\ %\qquad \text{since~$k \geq 2$}\\
&\leq C \left (\frac{8k}{n}\right ) ^{\frac{[s]}{2[s]+2}} +\frac{n}{k}\eta e^{-\gamma k}  \label{eq:beta_exp_diff4} \\%\qquad \text{since~$n-4k \geq n/2$ for $k \leq n/8$}\\
&\leq \frac{8Ck}{n^{\frac{[s]}{2[s]+2}}}   +\eta n e^{-\gamma k}
\label{eq:bound}
\end{align}
where 
\eqref{eq:beta_exp_diff00} follows from triangle inequality and observing that under $E$ the estimators $\hat{\beta}_N$ and $\hhat{\beta}_N$ are equal,
\eqref{eq:beta_exp_diff0} follows from \eqref{eq:ourb},
\eqref{eq:beta_exp_diff1} follows from
\eqref{eq:diffs},\eqref{eq:norm1} and \eqref{eq:norm2},
\eqref{eq:beta_exp_diff2} follows from observing that $\beta(k)\leq 1$ and the geometric ergodicity of the process, and \eqref{eq:beta_exp_diff3} and  \eqref{eq:beta_exp_diff4} follow from the definition of $N$ and the fact that $2 \leq k \leq \lfloor n/8\rfloor $. 
Optimizing \eqref{eq:bound} for $k$ we obtain

\begin{equation}\label{eq:kstar}
k^{\star}=\left \lfloor \frac{1}{\gamma}\left (\log  \frac{\gamma  \eta}{8C} + \left (\frac{3[s]+2}{2[s]+2}\right )\log n \right )\right \rfloor
\end{equation}
which in turn leads to 
\begin{align}
\E |\beta(m)-\hhat{\beta}_N(m)| 
 &\leq \frac{8Cn^{-\frac{[s]}{2[s]+2}}}{\gamma}\left ( e^\gamma+\log  \frac{\gamma  \eta}{8C} + \left (\frac{3[s]+2}{2[s]+2}\right )
 \log n  \right) 
\end{align}
This completes the proof of the bound on the expected error. 
For the high probability bound observe that \cite[Theorem 5.1.13]{nickl} states that  for all $t>0$,
\begin{equation}
\Pr\bigl(N \|\hhat{f}_{m,N}-\E(\hhat{f}_{m,N}) \|_1 \geq (3/2) N \E \|\hhat{f}_{m,N}-\E(\hhat{f}_{m,N}) \|_1 + \sqrt{2 N t} \|K\|_1 + t 5 \|K\|_1  \bigr)  \leq e^{-t}.
\end{equation}
Furthermore, $\E(\hhat{f}_{m,N}) = K_{h_N} * f_m$, where $K_h(x) = (1/h) K(x/h)$ for $h>0, x \in \R^2$. Hence,
\begin{equation} \label{eq:high_prob_cont}
\|\hhat{f}_{m,N}- f_m \|_1  \leq \|\hhat{f}_{m,N}-\E(\hhat{f}_{m,N}) \|_1 + \|K_{h_N} * f_m - f_m\|_1.
\end{equation}
The latter term can be bounded by using \cite[Proposition 4.3.33]{nickl},
\begin{equation} \label{eq:corr_sob}
\|K_{h_N} * f_m - f_m\|_1 \leq h_N^{[s]} \|f_m\|_{W_1^{[s]}(\R^2)},
\end{equation}
Note that there is a typo in \cite[Proposition 4.3.33]{nickl} which states the result as in the one-dimensional case. The inequality \eqref{eq:corr_sob} relies on the remainder term of a Taylor series. The remainder term in $\|\cdot\|_1$-norm is upper bounded by (Minkowski inequality for integrals)
\[
[s] h_N^{[s]}  \sum_{|\alpha| = [s]} \frac{1}{\alpha!} \int_0^1 (1-u)^{[s]-1} \| D^\alpha f_m\|_1  \, du  \leq 
[s] h_N^{[s]} \int_0^1 (1-u)^{[s]-1} du \|f_m\|_{W_1^{[s]}(\R^2)}, 
\]
where $\alpha =(\alpha_1,\alpha_2)$ is  multi-index of dimension $2$, $\alpha! = \alpha_1! \alpha_2!$, the integral $[s] \int_0^1 (1-u)^{[s]-1}$ is equal to $1$, and $D^\alpha f_m$ is a weak-derivative of $f_m : \R^2 \to \R$.
Hence, 
\begin{align*}
N \|\hhat{f}_{m,N}- f_m \|_1
\leq N \|\hhat{f}_{m,N}-\E(\hhat{f}_{m,N}) \|_1  + N \Lambda h_N^{[s]}  
\end{align*}
and with probability at least $1-e^{-u}$, 
\begin{align*}
N \|\hhat{f}_{m,N}- f_m \|_1
\leq N  \Lambda h_N^{[s]} + 
(3/2) N \E \|\hhat{f}_{m,N}-\E(\hhat{f}_{m,N}) \|_1 + \sqrt{2 N u} \|K\|_1 + u 5 \|K\|_1.
\end{align*}
Recall that 
$\|f_m\|_{W_1^{[s]}(\R^2)} \leq \|f_m\|_{B_{1,\infty}^2(\R^2)}$ and, therefore, from \eqref{eq:norm2} it follows that for any $f_m$ such that $\|f_m\|_{B_{1,\infty}^2(\R^2)}\leq \varLambda $, with probability $1-e^{-u}$ we have,
\[
 \|\hhat{f}_{m,N}- f_m \|_1
\leq   \Lambda h_N^{[s]} + 
(3/2)   \widetilde{C} N^{-\frac{[s]}{2[s]+2}}
+ \sqrt{2  u/N} \|K\|_1 + u 5 \|K\|_1/N.
\]
Substituting $h_N$ as stated in Condition \ref{ass}.ii, yields that with probability $1-e^{-u}$,
\begin{align*}
 &\|\hhat{f}_{m,N}- f_m \|_1   \\
& \leq \left((3/2)\widetilde{C} +  \varLambda (c\varLambda)^{-\frac{[s]^2}{[s]+1}} \left(\frac{[s]-1}{2}\right)^{-\frac{[s]^2}{2[s]+2}} \right)
N^{-\frac{[s]}{2[s]+2}}
+ \sqrt{2  u/N} \|K\|_1 + u 5 \|K\|_1/N.
\end{align*}
Similarly, we can bound the difference between $\hat f_{0,N}^*$ and $f_0$ in high probability; for $u > 0$, with probability $1-e^{-u}$, it follows from \eqref{eq:f0f0} that
\begin{align*}
 &\|f_0\otimes f_0-\hhat{f}_{0,N}\otimes \hhat{f}_{0,N}\|_1/ (1+c_0) \\ 
 &\leq \|\hhat{f}_{0,N}- f_0 \|_1 
  \\
&\leq\left((3/2)\widetilde{C} +  \varLambda (c\varLambda)^{-\frac{[s]^2}{[s]+1}} \left(\frac{[s]-1}{2}\right)^{-\frac{[s]^2}{2[s]+2}} \right)
N^{-\frac{[s]}{2[s]+2}}
+ \sqrt{2  u/N} \|K\|_1 + u 5 \|K\|_1/N.
\end{align*}
Substituting this into \eqref{eq:diffs} gives that with probability at least $1-2 e^{-u}$,
\begin{align}\label{eq:azaladdedthis}
&|\beta(m)-\hhat{\beta}_N(m)|/(1+c_0/2) \nonumber \\
&\leq \left((3/2)\widetilde{C} +  \varLambda (c\varLambda)^{-\frac{[s]^2}{[s]+1}} \left(\frac{[s]-1}{2}\right)^{-\frac{[s]^2}{2[s]+2}} \right)
N^{-\frac{[s]}{2[s]+2}}
+ \sqrt{2  u/N} \|K\|_1 + u 5 \|K\|_1/N.
\end{align}
Furthermore, with probability 
\[1-N \beta(k^{\star}) \geq 1 - 8 C  
n^{-\frac{[s]}{2[s] + 2}} 
\]
$\hhat{f}_{0,N} = \hat{f}_{0,N}$ and $\hhat{f}_{m,N} = \hat{f}_{m,N}$. By setting $u= \log(n) [s]/(2[s]+2)$ in \eqref{eq:azaladdedthis} we observe that with probability at least $1 - 8 C  
n^{-\frac{[s]}{2[s] + 2}} $
it holds that 
\begin{align*}
&|\beta(m)-\hat{\beta}_N(m)|/(1+c_0/2) \\
&\leq \left(\frac{3}{2}\widetilde{C} + \varLambda (c\varLambda)^{-\frac{[s]^2}{[s]+1}} \left(\frac{[s]-1}{2}\right)^{-\frac{[s]^2}{2[s]+2}} \right)
N^{-\frac{[s]}{2[s]+2}}
+ \sqrt{\frac{[s] \log(n)}{([s]+1)N}} \|K\|_1 + \frac{5[s] \log(n) |K\|_1}{(2[s]+2) N}.
\end{align*}
Finally, observing that $N \geq (n/4k^{\star}) -1$, and noting that 
$16N^{-\frac{[s]}{2[s]+2}} \geq (N+1)^{-\frac{[s]}{2[s]+2}}$ whenever $N \geq 2$ and $[s]\geq 1$ we obtain,
\begin{align*}
&|\beta(m)-\hat{\beta}_N(m)|  \\
&\leq
16 \bigl(1+ \frac{c_0}{2}\bigr) 
\Biggl(\frac{3}{2}\widetilde{C} +  \varLambda (c\varLambda)^{-\frac{[s]^2}{[s]+1}} \left(\frac{[s]-1}{2}\right)^{-\frac{[s]^2}{2[s]+2}}  \\
&\hspace{3cm} + \left(\sqrt{\frac{[s]}{([s]+1)}} + \frac{5[s] }{(2[s]+2)}\right) \log(n) |K\|_1  \Biggr)
\left(\frac{n}{4k^{\star}}\right)^{-\frac{[s]}{2[s]+2}}.
\end{align*}
Noting that $(4k^{\star})^\frac{[s]}{2[s]+2} \leq 4k^{\star}$ and letting
\[C_1  = \frac{1}{\gamma} \log\left(\frac{\gamma \eta}{8C}\right) \left(\frac{3}{2}\widetilde{C} +  \varLambda (c\varLambda)^{-\frac{[s]^2}{[s]+1}} \left(\frac{[s]-1}{2}\right)^{-\frac{[s]^2}{2[s]+2}}\right)\] and 
\[C_2=\frac{4\|K\|_1}{\gamma} \log\left(\frac{\gamma \eta}{8C}\right)  + \frac{3}{2\gamma} \left(\frac{3}{2}\widetilde{C} +  \varLambda (c\varLambda)^{-\frac{[s]^2}{[s]+1}} \left(\frac{[s]-1}{2}\right)^{-\frac{[s]^2}{2[s]+2}} \right),\]
we obtain 
\[
|\beta(m)-\hat{\beta}_N(m)| \leq 64 \bigl(1+ \frac{c_0}{2}\bigr) (C_1 + C_2 \log(n) + \frac{6\|K\|_1}{\gamma} \log^2(n)) n^{-\frac{[s]}{2[s]+2}}.\]
\end{proof}

\begin{proof}[Proof of Theorem~\ref{prop:beta_d}]
As in the proof of Theorem~\ref{lem:expected_beta_error}, we use a coupling argument together with concentration bounds on independent copies. More specifically, consider the geometrically ergodic Markov sample  $X_0,\dots,X_n$, and define the sequence of tuples $Z_i = ( X_{2i(k+1)},X_{2i(k+1)+m})$,~$i =0,1,2,\dots, N-1$ where $N=N(k,n):=\lfloor \frac{n-k}{2(k+1)} \rfloor $ for some fixed $k \in m+1,\dots,\lfloor n/8 \rfloor$; as in the continuous state-space setting, an optimal choice for $k$ is specified later in the proof, see \eqref{eq:kstar2}. 
As follows from Lemma~\ref{lem:ourberbee} there exists a sequence of independent random vectors $Z_i^*=(Z_{i,1}^*,Z_{i,2}^*)$ for $i=0,1,\dots,N-1$ each of which takes value in $\X^2$ and has the same distribution as $Z_i$ such that
\begin{equation}\label{eq:ourbd}
    \Pr\left (\{\exists i \in 0,\dots,N-1: Z_i^* \neq Z_i\}\right) \leq N\beta(k)
\end{equation} 
Define
\[
\hat{\beta}^*_{N}(m) := \frac{1}{2} \sum_{u \in \X}\sum_{v \in \X}|\hat{P}^*_{m,N}(\{(u,v)\})-\hat{P}^*_{0,N}(\{u\})\hat{P}^*_{0,N}(\{v\})|
\]
where 
$
\hat{P}^*_{m,N}((u,v)):=\frac{1}{N}\sum_{i=0}^{N-1}\mathbf{1}_{\{(u,v)\}}(Z^*_i)
$ 
and 
$ 
\hat{P}^*_{0,N}(u):=\frac{1}{N}\sum_{i=0}^{N}  \mathbf{1}_{\{u\}}(Z^*_{i,1})
$. 
By a simple application of Jensen's inequality and noting that the random variables $Z^*_i,~i=0,\dots,N-1$ are iid, for each $z \in \X \times \X$ we have, 
\begin{align}\label{eq:P_hat_0}
\E[|\hat{P}_{m,N}^*(z) - P_m(z)|] 
&\leq \frac{1}{N}(\sum_{i=0}^{N-1}\E (\mathbf{1}_{\{z\}}(Z^*_i) - \E\mathbf{1}_{\{z\}}(Z^*_i))^2)^{1/2} \leq \sqrt{P_m(z)/N}
\end{align}
where the second inequality is due to $\VARd{\mathbf{1}_{\{z\}}(Z^*_i) - \E\mathbf{1}_{\{z\}}(Z^*_i)} \leq P_m(z)$. 
Similarly, for each $u \in \X$ we obtain
$
\E[|\hat{P}_{0,N}^*(u) - P_0(u)|] 
\leq \sqrt{P_0(u)/N}
$.
It follows that 
\begin{align}
&2 \E|\hat{\beta}^*_{N}(m)-\beta(m)| \nonumber \\
&\leq \sum_{(u,v) \in \X^2} \E|P_m(\{(u,v)\}) - \hat{P}^*_{m,N}(\{(u,v)\})| + 2 \sum_{u\in \X} \E |P_0(\{u\}) - \hat{P}^*_{0,N}(\{u\})|\\
& \leq 2|\X| N^{-1/2}.
\end{align}
In much the same way as in the proof of Theorem~\ref{lem:expected_beta_error}, let $E:=\{Z_i^*= Z_i,~i \in 0,\dots,N-1\}$, and observe that $\E(|\hat{\beta}^*_{N}(m)-\hat{\beta}_{N}(m)|\;| E) = 0$. Moreover, recall that $k \leq \lfloor n/8 \rfloor$. 
We obtain, 
\begin{align}
\E |\beta(m)-\hat{\beta}_N(m)| 
&\leq\E |\beta(m)-\hhat{\beta}_N(m) | + 2 N \beta(k)  \label{eq:beta_exp_diff0d}\\
&\leq |\X| \left ( \frac{4k}{n-4k}\right)^{1/2}+2 n\eta e^{-\gamma k} \nonumber\\
& \leq |\X| \left ( \frac{8k}{n}\right)^{1/2}+2 n\eta e^{-\gamma k} \nonumber\\
&\leq \sqrt{8}|\X| n^{-1/2} k +2 n\eta e^{-\gamma k} \label{eq:boundkb}
\end{align}
where \eqref{eq:beta_exp_diff0} follows from \eqref{eq:ourbd}. Optimizing \eqref{eq:boundkb} we obtain 
\begin{equation}\label{eq:kstar2}
k^{\star} = \left \lfloor \frac{1}{\gamma}\log \left(\frac{\eta\gamma n^{3/2}}{\sqrt{2}|\X| }\right)\right \rfloor
\end{equation}
where $n$ is taken large enough so ensure that $k^{\star} \geq 2$ (see \eqref{eq:goodn} below).
This choice of $k^{\star}$ and $n$ leads to 
\begin{align*}
\E|\hat{\beta}_{N}(m)-\beta(m)|
&\leq \sqrt{8}|\X| n^{-1/2} k^{\star} +2 n\eta e^{-\gamma k^\star}\\
&\leq \sqrt{\frac{8}{n}}  \frac{|\X|}{\gamma}\log \left(\frac{\eta\gamma n^{3/2}}{\sqrt{2}|\X| }\right) +2 n\eta e^{-\gamma (\frac{1}{\gamma}\log \left(\frac{\eta\gamma n^{3/2}}{\sqrt{2}|\X| }\right)-1)}\\
&\leq \frac{\sqrt{8}|\X| n^{-1/2}}{\gamma}\left(e+\log \left(\frac{\eta\gamma}{\sqrt{2}|\X| }\right)+\frac{3}{2}\log n\right).
\end{align*}
Next, take $N=N(k^{\star},n) = \lfloor \frac{n-k^{\star}}{2(k^{\star}+1)} \rfloor$, with $k^{\star}$ given by \eqref{eq:kstar2} and let $n$ be large enough so that
\begin{align}\label{eq:goodn}
n^{3/2} \geq \max \left \{\frac{\sqrt{2}|\X| e^{2\gamma}}{\eta \gamma}, \frac{\eta \gamma}{|\X|} \right\}.
\end{align} 
Substituting for $k^{\star}$ and noting that 
\begin{align*}
N=\left \lfloor \frac{n-k^{\star}}{2(k^{\star}+1)}\right\rfloor 
%&\leq \left \lfloor \frac{n}{2(k^{\star}+1)}\right\rfloor\\
%&= \left \lfloor \frac{n}{2(\left \lfloor \frac{1}{\gamma}\log \left(\frac{\eta\gamma n^{3/2}}{\sqrt{2}|\X| }\right)\right \rfloor+1)}\right\rfloor\\
%&\leq \left \lfloor \frac{n}{2\frac{1}{\gamma}\log \left(\frac{\eta\gamma n^{3/2}}{\sqrt{2}|\X| }\right)}\right\rfloor\\
\leq \frac{\gamma n}{2\log \left(\frac{\eta\gamma n^{3/2}}{\sqrt{2}|\X| }\right)}
\end{align*}
we have,
\begin{align}\label{eq:Nbetak}
N\eta e^{-\gamma k^{\star }} 
\leq \frac{e\sqrt{2}|\X| n^{-1/2}}{2\log(\frac{\eta \gamma n^{3/2}}{\sqrt{2}|\X|}) }.
\end{align}
On the other hand, by Hoeffding's inequality, for any $\epsilon >0$ and each $u \in \X$ we have,
\begin{align}
\Pr(|\hat{P}^*_{0,N}(\{u\}) - P_0(\{u\})| \geq \frac{\epsilon}{2|\X|})
& \leq 2\exp\left \{-\frac{N\epsilon^2}{2|\X|^2}\right \}
\end{align}
Similarly, for each $(u,v) \in \X^2$ it holds that 
\begin{align}
\Pr(|\hat{P}^*_{m,N}(\{(u,v)\}) - P_m(\{(u,v)\})| \geq \frac{\epsilon}{2|\X|^2})
& \leq 2\exp\left \{-\frac{N\epsilon^2}{2|\X|^4}\right \}
\end{align}
It follows that 
\begin{align}
\Pr(|\hat{\beta}^*_{N}(m)&-\beta(m)|  \geq \epsilon/2) \nonumber \\
&\leq \sum_{u,v} \Pr(|P_m(\{(u,v)\}) - \hat{P}^*_{m,N}(\{(u,v)\})| \geq \frac{\epsilon}{2|\X|^2}) \nonumber \\
&\qquad\qquad\qquad+
2 \sum_{u}\Pr(|P_0(\{v\}) - \hat{P}^*_{0,N}(\{u\})| \geq \frac{\epsilon}{2|\X|})\nonumber \\
&\leq 2|\X|^2 \exp\left \{-\frac{N\epsilon^2}{2|\X|^4}\right \}+2|\X|\exp\left \{-\frac{N\epsilon^2}{2|\X|^2}\right \} \nonumber \\
&\leq  4|\X|^2 \exp\left \{-\frac{N\epsilon^2}{2|\X|^4}\right \}\nonumber \\
&\leq 4|\X|^2\exp\{-\frac{(n-4k^{\star})\epsilon^2}{8k^{\star}|\X|^4}\}\label{eq:boundd0}\\
&\leq 4|\X|^2\exp\{-\frac{n\epsilon^2}{16k^{\star}|\X|^4}\} \label{eq:bound1}\\
&=  4|\X|^2\exp\left \{-\frac{\gamma n\epsilon^2}{16 |\X|^4 \left(\log \left(\frac{\eta\gamma}{\sqrt{2}|\X| }\right) + \frac{3}{2}\log n\right)}\right\}\label{eq:indep_bound0} \\
&\leq 4|\X|^2\exp\{-\frac{\gamma n\epsilon^2}{48 |\X|^4 \log n}\}\label{eq:indep_bound}
\end{align}
where, \eqref{eq:boundd0} follows from the choice of $N=\lfloor \frac{n-k^{\star}}{2(k^{\star}+1)}\rfloor $ and noting that $k^{\star} \geq 2$,
\eqref{eq:bound1} follows from recalling that in general, $k$ (and thus also $k^{\star}$), is less than $\lfloor n/8\rfloor $, and finally, \eqref{eq:indep_bound0}  and \eqref{eq:indep_bound} follow from substituting the value of $k^{\star}$ as given by \eqref{eq:kstar} and  observing that by \eqref{eq:goodn} we have $
\frac{3}{2}\log n \geq \log(\frac{\eta \gamma }{|\X|})$.
Hence, by \eqref{eq:ourbd}, \eqref{eq:Nbetak} and \eqref{eq:indep_bound} we obtain,
\begin{align*}
\Pr(|\hat{\beta}_{N}(m)-\beta(m)| \geq \epsilon ) 
&\leq 
N\beta(k^{\star}) + 
\Pr(|\hat{\beta}^*_{N}(m)-\beta(m)| \geq \epsilon/2 ) \\
%&\leq N\eta e^{-\gamma k^{\star}} + 4|\X|^2 \exp\left \{-\frac{N\epsilon^2}{2|\X|^4}\right \}\\
&\leq  \frac{e\sqrt{2}|\X| n^{-1/2}}{2\log\left (\frac{\eta \gamma n^{3/2}}{\sqrt{2}|\X|}\right) }+4|\X|^2\exp\left \{-\frac{\gamma n\epsilon^2}{48 |\X|^4 \log n}\right \}
\end{align*}
\end{proof}
\begin{proof}[Proof of Theorem~\ref{prop:beta_dvc}]
As in the proof of Theorem~\ref{prop:beta_d}, we start by a coupling argument, with the difference that instead of generating $2$-tuples, we generate blocks of length $k+1$ for an appropriate value of $k$ which we specify further in the proof. Specifically, given $X_0,\dots,X_n$  define \[\tilde{Z}_i = ( X_{2i(k+1)},X_{2i(k+1)+1},\dots,X_{(2i+1)k+2i})\]
for $i =0,1,\dots, N-1$ where $N=N(k,n):=\lfloor \frac{n-k}{2(k+1)} \rfloor $ for some fixed $k \in 1,\dots,\lfloor n/8 \rfloor$; an optimal choice for $k$ is specified later in the proof, see \eqref{eq:kstar2_vc}. 
By Lemma~\ref{lem:ourberbee} there exists a sequence of independent random vectors $\tilde{Z}_i^{*}=(\tilde{Z}_{i,0}^{*},\dots,\tilde{Z}_{i,k}^{*})$ for $i=0,1,\dots,N-1$ each of which takes value in $\X^{k+1}$ and has the same distribution as $\tilde{Z}_i$ such that
\begin{equation}\label{eq:ourbdvc}
    \Pr\left (\{\exists i \in 0,\dots,N-1: \tilde{Z}_i^{*} \neq \tilde{Z}_i\}\right) \leq N\beta(k).
\end{equation}
Define
\[
\hat{\beta}^{\dagger}_{N}(m) :=\sum_{u \in \X}\sum_{v \in \X}|\hat{P}^{\dagger}_{m,N}(\{(u,v)\})-\hat{P}^{\dagger}_{0,N}(\{u\})\hat{P}^{\dagger}_{0,N}(\{v\})|
\]
where 
$
\hat{P}^{\dagger}_{m,N}((u,v)):=\frac{1}{N}\sum_{i=0}^{N-1}\mathbf{1}_{\{(u,v)\}}(\tilde{Z}^{*}_{i,0},\tilde{Z}^{*}_{i,m})
$ 
and 
$ 
\hat{P}^{\dagger}_{0,N}(u):=\frac{1}{N}\sum_{i=0}^{N}  \mathbf{1}_{\{u\}}(\tilde{Z}^{*}_{i,0})
$. 
As in the proof of Theorem~\ref{prop:beta_d}, for each $u \in \X$ it holds that 
\begin{align}\label{eq:P_hat_0_vc}
\E[|\hat{P}_{0,N}^{\dagger}(u) - P_0(u)|] \leq \sqrt{P_0(u)/N} \leq N^{-1/2}.
\end{align}
Define the class of indicator functions
\[
\mathcal H_k = \{h_{m,z}: \X^{k+1} \to \{0,1\}: z \in \X \times \X,~h_{m,z}(\x):=\mathbf{1}_{\{z\}}(x_0,x_m),~m=1,\dots,k \}.
\]
One can verify that the VC dimension of $\mathcal{H}_k$ is at most $\log_2(|\mathcal{X}| k)$. To see this, note first that the indicator functions in $\mathcal{H}_k$ depend on only two coordinates at a time, rather than on all coordinates simultaneously. Consequently, the VC dimension of $\mathcal{H}_k$ will typically exceed one but remain limited by the combinatorial structure of these pairwise tests.
The upper bound follows from the standard result that the VC dimension of a finite hypothesis class $\mathcal{H}_k$ is at most $\log_2 |\mathcal{H}_k|$. In our case, the class $\mathcal{H}_k$ has cardinality $|\mathcal{X}|^2 k$. However, since the first coordinate can be fixed to a constant value (say, $1$), we can restrict attention to a subclass of functions of size $|\mathcal{X}| k$ that can shatter as many points as $\mathcal{H}_k$. Therefore, the VC dimension of $\mathcal{H}_k$ is bounded by $\log_2 (|\mathcal{X}| k)$.

Therefore, as follows from \cite[pp. 217]{nickl} it holds that,
\begin{align}\label{eq:P_hat_m_vc}
\E[\sup_{\substack{m \in 1,\dots,k\\(u,v) \in \X^2}} |P^{\dagger}_{m,N}((u,v)) - P_m((u,v))|] 
&\leq \sqrt{\frac{8\log_2(|\X| k)\log N}{N}}
\end{align}
By \eqref{eq:P_hat_0_vc} and \eqref{eq:P_hat_m_vc} we have,
\begin{align}
&2\E[\sup_{m \in 1,\dots,k}|\hat{\beta}^{\dagger}_{N}(m)-\beta(m)|]\nonumber \\
&\leq \E[ \sup_{m \in 1,\dots k} \sum_{(u,v) \in \X^2}|P_m(\{(u,v)\}) - \hat{P}^{\dagger}_{m,N}(\{(u,v)\})| + 2 \sum_{u\in \X} |P_0(\{u\}) - \hat{P}^{\dagger}_{0,N}(\{u\})|]\nonumber \\
&\leq |\X|^2 \E[ \sup_{\substack{m \in 1,\dots, k\\(u,v) \in \X^2}}  |P_m(\{(u,v)\}) - \hat{P}^{\dagger}_{m,N}(\{(u,v)\})|] + 2 \sum_{u\in \X} \E|P_0(\{u\}) - \hat{P}^{\dagger}_{0,N}(\{u\})|\nonumber \\
& \leq |\X|^2 \sqrt{\frac{8\log_2(|\X| k)\log N}{N}} + 2|\X|N^{-1/2}\nonumber \\
&\leq 2|\X|^2 \sqrt{\frac{8\log_2(|\X| k)\log N}{N}}\label{eq:beta_star_beta_vc}
\end{align}
%Since $N =\lfloor\frac{n-k}{2(k+1)} \rfloor \geq \frac{n-k}{2(k+1)}-1$ taking  $n \geq \frac{2k(k+1)}{3k+2}$ ensures that $N \geq k$.  
Take  $n \geq \frac{2k(k+1)}{3k+2}$. From  the coupling argument given earlier we obtain,
\begin{align}
\E[\sup_{m \in 1,\dots,k}|\hat{\beta}_{N}(m)-\beta(m)|] 
&\leq 
\E[\sup_{m \in 1,\dots,k}|\hat{\beta}^{\dagger}_{N}(m)-\beta(m)|]+N\beta(k)\nonumber \\
&\leq |\X|^2\sqrt{\frac{8\log_2(|\X| k)\log N}{N}}+N\eta e^{-\gamma k} \label{eq:vclem1} \\
&\leq |\X|^2\sqrt{\frac{8}{N}}\log_2(N|\X|) + N\eta e^{-\gamma k}\label{eq:vclem2} \\
&\leq 4|\X|^2 \sqrt{\frac{2k}{n}}\log_2 (n|\X|)+n\eta e^{-\gamma k}\label{eq:vclem3}\\
&\leq 4|\X|^2 \sqrt{\frac{2}{n}} \log_2(n|\X|) k+n\eta e^{-\gamma k}
\label{eq:bound_vc}
\end{align}
where \eqref{eq:vclem1} follows from \eqref{eq:beta_star_beta_vc}, and \eqref{eq:vclem2} follows from noting that $N = N(k,n) \geq k$ for $n \geq \frac{2k(k+1)}{3k+2}$; similarly, \eqref{eq:vclem3} follows form the choice of $N=N(k,n) \leq \frac{n-k}{2(k+1)}-1$ and noting that $k \geq 1$. 
Optimizing \eqref{eq:bound_vc} we have
\begin{align}\label{eq:kstar2_vc}
k^{\dagger} = \left \lfloor \frac{1}{\gamma}\log\left( \frac{\eta \gamma n^{3/2}}{4\sqrt{2}|\X|^2\log_2(n|\X|)}\right)\right \rfloor 
\end{align}
with 
\begin{align}\label{eq:goodnvc}
n \geq \max \left \{\frac{4\sqrt{2}|\X|^3 e^\gamma}{\eta \gamma}, \frac{2k(k+1)}{3k+2}\right \}.
\end{align}  This leads to
\begin{align}
\E[\sup_{m \in 1,\dots,k}&|\hat{\beta}_{N}(m)-\beta(m)|]\nonumber \\ 
&\leq \frac{4\sqrt{2}|\X|^2 n^{-1/2} \log_2 (n|\X|)}{\gamma}\left(e+\log\left(\frac{\eta \gamma n^{3/2}}{4\sqrt{2}|\X|^2\log_2(n|\X|)}\right)\right).
\end{align}
Take $N=N(k^{\dagger},n) = \lfloor \frac{n-k^{\dagger}}{2(k^{\dagger}+1)} \rfloor$, with $k^{\dagger}$ given by \eqref{eq:kstar2_vc}.
It follows that,
\begin{align}\label{eq:Nbetakvc}
N\eta e^{-\gamma k^{\dagger }} 
\leq \frac{2e\sqrt{2}|\X|^2 \log_2(n|\X|)n^{-1/2}}{\log(\frac{\eta \gamma n^{3/2}}{4\sqrt{2}|\X|^2 \log_2(n|\X|)})}
\end{align}
On the other hand, 
by Hoeffding's inequality, for any $\epsilon >0$ and $u \in \X$ it holds that
\begin{align}\label{eq:hoeff_vc}
\Pr(|\hat{P}^{\dagger}_{0,N}(\{u\}) - P_0(\{u\})| \geq \frac{\epsilon}{2|\X|})
& \leq 2\exp\left \{-\frac{N\epsilon^2}{2|\X|^2}\right \}
\end{align}
Furthermore, noting that $\mathcal H_{k^{\dagger}}$ is a VC-class, by  \cite[Theorem 12.5]{devroye2013probabilistic} for $\epsilon >0$ 
we have, \begin{align}\label{eq:vc}
\Pr(\sup_{\substack{m\in 1,\dots,k^{\dagger}\\ (u,v) \in \X^2}}
|\hat{P}^{\dagger}_{m,N}(\{(u,v)\}) - P_m(\{(u,v)\})| \geq \frac{\epsilon}{2|\X|^2})
 \leq 8\log_2(|\X|k)e^{-\frac{N\epsilon^2}{128|\X|^4}}
\end{align}
Furthermore,  noting that in general all $k$ are taken to be less than  $\lfloor n/8\rfloor $, a minor manipulation gives, 
\begin{align}
    N%&=\lfloor \frac{n-k}{2(k+1)}\rfloor \nonumber \\
    &\geq \frac{n-k}{2(k+1)}-1\nonumber \\
    &\geq \frac{n-4k}{4k} &\text{since $k\geq 2$}\nonumber \\
    &\geq \frac{n}{8k}\label{eq:vclemN}
\end{align}
where \eqref{eq:vclemN} is from noting that $k \leq \lfloor \frac{n}{8}\rfloor$. 
Therefore, for any $\epsilon >0$ we have,
\begin{align}
\Pr(\sup_{m \in 1,\dots,k^{\dagger}}&|\hat{\beta}^{\dagger}_{N}(m)-\beta(m)|  \geq \epsilon/2) \nonumber \\
&\leq |\X|^2 \Pr\left(\sup_{\substack{m \in 1,\dots,k^{\dagger}\\(u,v)\in \X^2}}|P_m(\{(u,v)\}) - \hat{P}^{\dagger}_{m,N}(\{(u,v)\})| \geq \frac{\epsilon}{2|\X|^2}\right ) \nonumber \\
&\qquad\qquad\qquad+
2 \sum_{u}\Pr\left (|P_0(\{v\}) - \hat{P}^{\dagger}_{0,N}(\{u\})| \geq \frac{\epsilon}{2|\X|}\right )\nonumber \\
&\leq 8|\X|^2 \log(|\X|k^{\dagger}) \exp\left \{-\frac{N\epsilon^2}{128|\X|^4}\right \}+2|\X|\exp\left \{-\frac{N\epsilon^2}{2|\X|^2}\right \} \label{eq:indep_bound_vc_1}\\
&\leq  16|\X|^2 \log(|\X|k^{\dagger}) \exp\left \{-\frac{N\epsilon^2}{128|\X|^4}\right \}\nonumber\\
&\leq 16|\X|^2 \log(|\X|k^{\dagger}) \exp\{-\frac{n\epsilon^2}{1024 k^{\dagger}|\X|^4}\}\label{eq:indep_bound_vc_4}\\
&\leq 16|\X|^2 \log\left(\frac{3|\X|}{2\gamma}\log(\eta^{2/3} \gamma^{2/3} n)\right)\exp\left\{-\frac{\gamma n\epsilon^2}{3072|\X|^4\log n}\right\}\label{eq:indep_bound_vc}
\end{align}
where, \eqref{eq:indep_bound_vc_1} follows from \eqref{eq:hoeff_vc} and \eqref{eq:vc}, \eqref{eq:indep_bound_vc_4} follows from 
the choice \eqref{eq:vclemN}, and \eqref{eq:indep_bound_vc} follows from substituting for the value of $k^{\dagger}$ as given by \eqref{eq:kstar2_vc} and observing that by \eqref{eq:goodnvc} we have $
{ \frac{n^{3/2}}{\log(n|\X|)}} \geq { \frac{n}{|\X|}}\geq \frac{\eta \gamma }{8\sqrt{2}|\X|^2}$.
Hence, by \eqref{eq:ourbdvc}, \eqref{eq:Nbetakvc} and \eqref{eq:indep_bound_vc} we obtain,
\begin{align*}
\Pr(&\sup_{m=1,\dots,k^{\dagger}}|\hat{\beta}_{N}(m)-\beta(m)| \geq \epsilon ) \nonumber\\
&\leq 
N\eta e^{-\gamma k^{\dagger}} + 
\Pr(\sup_{m=1,\dots,k^{\dagger}}|\hat{\beta}^{\dagger}_{N}(m)-\beta(m)| \geq \epsilon/2 ) \\
&\leq  \frac{4\sqrt{2}|\X|^2 \log(n|\X|)n^{-1/2}}{\log(\frac{\eta \gamma n^{3/2}}{8\sqrt{2}|\X|^2})}+16|\X|^2 \log\left(\frac{3|\X|}{2\gamma}\log(\eta^{2/3} \gamma^{2/3} n)\right)\exp\left\{-\frac{\gamma n\epsilon^2}{3072|\X|^4\log n}\right\}
\end{align*}
and the result follows. 
\end{proof}


\end{document}